\documentclass[11pt,a4paper]{amsart}
\usepackage[foot]{amsaddr}

\usepackage{caption}
 \usepackage{geometry}
\usepackage{amsmath,amsfonts,amssymb,amsthm,mathrsfs}
\usepackage{graphicx}
\usepackage{stmaryrd}
\usepackage{hyperref}
\usepackage{tikz}
\usepackage{dsfont}

\newtheorem{theorem}{Theorem}[section]
\newtheorem{lemma}[theorem]{Lemma}
\newtheorem{proposition}[theorem]{Proposition}
\newtheorem{corollary}[theorem]{Corollary}

\newtheorem{example}[theorem]{Example}
\newtheorem{question}[theorem]{Question}
\newtheorem{condition}[theorem]{Condition}

\theoremstyle{remark}
\newtheorem{remark}[theorem]{Remark}

\numberwithin{equation}{section}

\newcommand{\R}{\mathbb{R}}
\newcommand{\Z}{\mathbb{Z}}
\def\P{\mathbb{P}} 
\newcommand{\E}{\mathbb{E}}

\newcommand{\id}{\mathds{1}}

\newcommand{\eps}{\varepsilon}
\newcommand{\cross}{\text{\textup{Cross}}}

\newcommand{\arm}{\text{\textup{Arm}}}

\DeclareMathOperator*{\esssup}{ess\,sup}

\begin{document}

\title[Percolation of the excursion sets of planar symmetric shot noise fields]{Percolation of the excursion sets of \\ planar symmetric shot noise fields}
\author{Raphael Lachieze-Rey$^1$}
\address{$^1$Universit\'{e} Paris Descartes / Universit\'e de Paris}
\email{raphael.lachieze-rey@parisdescartes.fr}
\author{Stephen Muirhead$^2$}
\email{smui@unimelb.edu.au}
\address{$^2$School of Mathematical Sciences, Queen Mary University of London (Current address: School of Mathematics and Statistics, University of Melbourne)}
\begin{abstract}
We prove the existence of phase transitions in the global connectivity of the excursion sets of planar symmetric shot noise fields. Our main result establishes a phase transition with respect to the level for shot noise fields with symmetric log-concave mark distributions, including Gaussian, uniform, and Laplace marks, and kernels that are positive, symmetric, and have sufficient tail decay. Without the log-concavity assumption we prove a phase transition with respect to the intensity of positive marks.
\end{abstract}
\thanks{The authors would like to thank Michael McAuley and Hugo Vanneuville for helpful discussions, and an anonymous referee for detailed feedback that greatly improved the presentation of the paper.}
\date{\today}

\maketitle

{\bf Keywords:} Percolation, excursion sets, shot noise fields, phase transition

{\bf AMS:}  60G60, 60K35

\section{Introduction}

Let $g(x) \in L^1(\mathbb{R}^2)$ be a continuous function and let $\mu$ be a distribution on~$\mathbb{R}$ with finite mean. The \textit{planar shot noise field} with \textit{kernel} $g$ and \textit{mark distribution} $\mu$ is defined as 
\begin{equation}
\label{e:ssn}
f(x) = \sum_{i \in \mathcal{P}} Y_i \, g(x-i)
\end{equation}
where $\mathcal{P}$ is a Poisson point process on $\mathbb{R}^2$ with unit intensity, and $(Y_i)_{i \in \mathcal{P}}$ are i.i.d.\ random variables drawn from the distribution $\mu$; the sum in \eqref{e:ssn} is well-defined almost surely by the integrability of $g$ and $\mu$. For example, one could take $g(x) = (1+|x|)^{-\alpha}, \alpha > 2$, and $\mu$ to be a normal distribution, a uniform distribution, or the Rademacher distribution. If $\mu$ is symmetric we call $f$ a \textit{planar symmetric shot noise field}.

\smallskip
The shot noise field in \eqref{e:ssn} is the Euclidean case of a general abstract class of infinitely divisible fields obtained by convoluting a kernel over a Poisson random measure. Such fields were introduced over a century ago by Campbell \cite{campbell} to model thermionic noise, and since then have been used, under several different names, in diverse fields such as image analysis \cite{BieDes,BieDesEuler} and telecommunications networks \cite{Bac10,BB15}. In the latter case, the points of the underlying Poisson process can be seen as emitters of an electromagnetic signal, while the field represents the total signal at every location of the space; see \cite{Bac10} for a detailed mathematical study of theory and applications of such models. At high frequency, i.e.\ when the density of points is high compared to the scale of the kernel, shot noise fields are also a good approximation of Gaussian fields with the same covariance structure. 

\smallskip
In this paper we are interested in the global connectivity of the (upper-)excursion sets
\[  \{f +\ell \ge 0\}  = \{ x \in \mathbb{R}^2 : f(x) + \ell \ge 0 \}  \ , \quad \ell \in \mathbb{R} , \]
of planar symmetric shot noise fields. In the analysis of telecommunications networks, the percolative properties of $ \{f + \ell \ge 0\} $ are of high importance in analysing the connectivity of the network \cite{BB15}. More generally, the geometric properties of $ \{f + \ell \ge 0\} $ have been the focus of many other studies \cite{BieDesEuler,Lr19,BST}. 

\smallskip By analogy with other planar percolation models (e.g.\ Bernoulli percolation on the square lattice \cite{har60,kes80}, level set percolation of planar stationary Gaussian fields \cite{alex96,bg17, rv17b,mv18}), it is natural to expect that the connectivity of $ \{f + \ell \ge 0\} $ undergoes a phase transition in the level $\ell$, from a subcritical phase $\ell < \ell_c$ in which all the components are bounded, to a supercritical phase $\ell > \ell_c$ in which $ \{f + \ell \ge 0\} $ has a unique unbounded connected component. It is furthermore natural to expect that $\ell_c = 0$ since the excursion sets are \textit{self-dual} at this level (i.e.\ $ \{f \ge 0\} $ is equal in law to $\{f \le 0 \}$). Our main result verifies this under general conditions on the kernel $g$ and symmetric mark distribution~$\mu$.

\subsection{The phase transition with respect to the level}
Let us begin by introducing the assumptions under which we work (see Section \ref{s:optimal} for remarks on their optimality). We say that a distribution $\mu$ has a \textit{sub-exponential tail} if there exists a $c > 0$ such that $\mu ( (-\infty,-x] \cup [x, \infty )) < e^{-cx}$ for sufficiently large $x$.
\begin{condition} $\,$
\label{c:1}
\begin{itemize}
\item (Mark) The mark distribution $\mu$ is symmetric and has a sub-exponential tail.
\item (Kernel; depends on a parameter $\alpha > 3$) The kernel $g$ is positive, is in $C^2(\mathbb{R}^2)$, and is symmetric with respect to reflection in the coordinate axes and rotation by $\pi/2$. Moreover, there exists a parameter $\alpha >3$ and a constant $c > 0$ such that, for every multi-index $k$ such that $|k| \le 2$ and $x \in \R^2$, 
\[ |\partial^{k} g(x)| < c(1+|x|)^{-\alpha-|k|} .\]
\end{itemize}
\end{condition}

Among other things, Condition \ref{c:1} implies certain regularity properties of the shot noise field $f$. In particular, since $g$ is in $C^2(\R^2)$, $f$ can be viewed as an (almost surely) absolutely convergent series in the Sobolev space $W^{2,1}$ on compact sets, which ensures that the sample paths of $f$ are (almost surely) almost everywhere twice differentiable. Then, since $x\mapsto \sup_{y:|y-x| \le 1} |\partial^{k} g(y)|$ is integrable for $|k|\le 2$,  $f$ is actually $C^2$-smooth \cite[Proposition 2.2.3]{Bac10a}. 

\smallskip
On the other hand, these conditions do not on their own guarantee that $f$ and its derivatives have bounded density. We shall impose this as an additional condition:
\begin{equation}
\label{a:bd}
\tag{BD}
\text{The random vector $(f(0),\nabla f(0))$ has a bounded density.}
\end{equation}
An obvious necessary condition for \eqref{a:bd} is that $\mu$ is non-zero and $g$ has unbounded support.

\smallskip
Finally, we introduce a crucial log-concavity assumption:
\begin{equation}
\label{a:lc}
\tag{LC}
\text{The mark distribution $\mu$ is log-concave, i.e.\ its density exists and is log-concave.}
\end{equation}

\noindent If $\mu$ is symmetric, then $\eqref{a:lc}$ implies that $\mu$ has a sub-exponential tail (Lemmas \ref{l:mr} and \ref{l:ed})

\smallskip
We are now ready to state our main result, establishing the existence of a phase transition in the connectivity of the excursion sets with respect to the level:

\begin{theorem}
\label{t:main1}
Suppose that $f$ satisfies Condition \ref{c:1} and~\eqref{a:bd}. Then
\begin{enumerate}
\item For every $\ell \le 0$, $ \{f + \ell \ge 0\} $ contains only bounded connected components almost surely. In particular, for every $\ell \in \mathbb{R}$, the $\ell$-level lines (i.e.\ the connected components of $\{f = \ell \}$) are almost surely bounded. 
\end{enumerate}
Suppose in addition \eqref{a:lc} holds. Then
\begin{enumerate}\addtocounter{enumi}{1}
\item For every $\ell > 0$, $ \{f + \ell \ge 0\} $ contains a unique unbounded connected component almost surely.
\end{enumerate}
Moreover, under Condition \ref{c:1} and~\eqref{a:bd} we have the following quantitative estimates on the connectivity of $\{f \ge 0 \}$ (formal statements are given in Section \ref{s:zero}):
\begin{enumerate}\addtocounter{enumi}{2}
\item For all $\rho >0$ there are $0<c_{-} \le c_{+}<1 $ such that, for all $R \ge 1$,
\begin{align*}
c_{-} \le \mathbb{P} \big( \{f \ge 0\} \text{\rm{ crosses horizontally the rectangle }}[0,R]\times [0,\rho R] \big) \le c_{+}.
\end{align*}
\item There are $c,c_{\text{\rm{Arm}}}>0$ such that, for all $1 \le r \le R$,
\begin{align*}
\mathbb{P} \big( \{ f \ge 0 \} \text{\rm{ connects $B(r)$ to $\partial B(R)$}} \big) \le c\left(
\frac{r}{R} \right)^{c_{\text{\rm{Arm}}}}
\end{align*}
where $B(r)$ denotes the closed ball of radius $r > 0$ centred at the origin.
\end{enumerate}
\end{theorem}

Let us give some concrete examples of shot noise fields to which Theorem \ref{t:main1} applies:

\begin{example}
If one takes the kernels
\[  g(x) = (1 + |x|^{2})^{-\alpha/2} , \ \alpha > 3 ,  \qquad \text{or} \qquad g(x) = \exp( -  ( 1+|x|^{2}) ^{\beta/2}) , \  \beta \in (0, 1) , \]
then the conclusions of Theorem \ref{t:main1} hold for any symmetric log-concave mark distribution (e.g.\ a centred normal, uniform, or Laplace distribution).
\end{example}
Since symmetric log-concave distributions have a sub-exponential tail (Lemmas \ref{l:mr} and \ref{l:ed}), it is immediate that Condition \ref{c:1} is satisfied by these examples, and see Appendix \ref{sec:density-power} for a proof that \eqref{a:bd} is satisfied.

\begin{remark}
\label{rk:kernel-motiv}
These examples are chosen as  the most natural parametric families of smooth kernels used in applications, and   also illustrate the boundaries of the validity of the assumptions, i.e.\ when the kernel decays too slowly and does not offer sufficient asymptotic independence ($ ( 1+ | x |^{2}) ^{-\alpha /2}$ for $\alpha \le 3$) or too fast to have bounded density ($\exp(- (1+ | x |^{2} )^{\beta/2 })$ for $\beta \ge 1$). The proof  in Appendix \ref{sec:density-power} can be applied to kernels  of the form $h( | x | )$ if the decays of $h$ and $h'$ are not too fast or not too slow, as above.
\end{remark}

\subsection{The phase transition with respect to the intensity of positive marks}
The log-concavity assumption \eqref{a:lc} is not satisfied in many examples of interest, such as the case of Rademacher marks ($\mu = (\delta_{-1} + \delta_1)/2$), and so Theorem \ref{t:main1} does not settle whether $\ell_c = 0$ in these examples.

\begin{question}
\label{q:1}
Does $ \{f + \ell \ge 0\} $ contain a unique unbounded connected component almost surely for every $\ell > 0 $ assuming only Condition \ref{c:1} and \eqref{a:bd}?
\end{question}

Nevertheless, the arguments we use to establish Theorem~\ref{t:main1} do imply the existence of a phase transition at $\ell = 0$ with respect to \textit{increasing the intensity of positive marks}. This holds for arbitrary mark distribution, and requires only slightly stronger assumptions on the decay of the kernel.

\smallskip
To state this result, observe that we can represent the shot noise field $f$ in \eqref{e:ssn} as
\[  f(x) \stackrel{d}{=} \sum_{i \in \mathcal{P}}  \Big( |Y_i| \id_{\{ \mathcal{U}_i \le 1/2 \}}  -  |Y_i|  \id_{\{ \mathcal{U}_i \ge  1/2 \}} \Big) g(x-i)   \]
where $(\mathcal{U}_i)_{i \in \mathcal{P}}$ are independent $[0,1]$-uniform random variables. Then for $\eta \in \mathbb{R}$ define the shot noise field
\begin{equation}
\label{e:feta}
 f^\eta(x) = \sum_{i \in \mathcal{P}} \Big( |Y_i| \id_{\{ \mathcal{U}_i \le 1/2 + \eta  \}}  -  |Y_i|  \id_{\{ \mathcal{U}_i \ge  1/2 + \eta \}} \Big) g(x-i) .
 \end{equation}
By the thinning property of Poisson point processes, $\eta$ can be viewed as parameterising the \textit{intensity of the positive marks}. If we assume $g \ge 0$, then $f^\eta$ is increasing in $\eta$.

\begin{theorem}
\label{t:main2}
Suppose that $f$ satisfies Condition \ref{c:1} for a parameter $\alpha > 4$. Moreover, suppose there exists $\eta_0 > 0$ such that \eqref{a:bd} holds for the field $f^\eta$ for all  $\eta \in [0,\eta _{0}]$. Then:
\begin{enumerate}
\item For $\eta \le 0$, $ \{f^\eta \ge 0\} $ contains only bounded connected components almost surely. 
\item For $\eta > 0$, $ \{f^\eta \ge 0\} $ contains a unique unbounded connected component almost surely.
\end{enumerate}
\end{theorem}

\begin{example}
If one takes the kernels
\[  g(x) = (1 + |x|^{2})^{-\alpha/2} , \ \alpha > 4,  \qquad \text{or} \qquad g(x) = \exp( -  ( 1+|x|^{2}) ^{\beta/2}) , \  \beta \in (0, 1) , \]
then the conclusions of Theorem \ref{t:main2} hold for any (non-zero) mark distribution with a sub-exponential tail, and in particular any with compact support (e.g.\ a Rademacher distribution).
\end{example}

Condition \ref{c:1} is obviously satisfied by this example, and see Appendix \ref{sec:density-power} for a proof that \eqref{a:bd} is satisfied for $f^\eta$ for all $\eta \in \mathbb{R}$.

\begin{remark}
Instead of demanding that \eqref{a:bd} holds for $f^\eta$ for all $\eta \in [0, \eta_0]$, one could ask that \eqref{a:bd} holds for the \textit{thinned} shot noise field
\[    f^{(\lambda)}(x) =    \sum_{i \in \mathcal{P}}  Y_i \id_{\{ \mathcal{U}_i \le \lambda  \}}   g(x-i) \]
for some $\lambda < 1$, i.e.\ the symmetric shot noise defined as in \eqref{e:ssn} except with a Poisson point process of intensity $\lambda < 1$. This is since 
\begin{align*}
    f^{\eta}(x)   \stackrel{d}{=}   \sum_{i \in \mathcal{P}}  \Big( Y_i \id_{\{ \mathcal{U}_i \le 1 -  2\eta  \}}  +   |Y_i|  \id_{\{ \mathcal{U}_i \ge  1  - 2\eta \}}  \Big) g(x-i)  
\end{align*}
and hence, by the thinning property of Poisson point processes, $f^\eta \stackrel{d}{=} f^{(1-2\eta)} + \tilde{f}$ where $\tilde{f}$ is an independent shot noise field. Therefore, since convolution does not increase the sup-norm of a function, the density of $(f^\eta(0), \nabla f^\eta(0))$ is bounded by the density of $(f^{(1-2\eta)}(0), \nabla f^{(1-2\eta)}(0))$, and, by a further application of the thinning property, also bounded by the density of $(f^{(\lambda)}(0), \nabla f^{(\lambda)}(0))$ for any $\lambda \le 1 - 2\eta$.
\end{remark}

\subsection{Relationship to previous work}
Although all the statements in Theorems \ref{t:main1} and~\ref{t:main2} are new for planar shot noise fields, some of them could be deduced from previous work on related models. Indeed, the fact that $ \{f \ge 0\}$ has bounded connected components could be deduced from general results on dependent planar percolation models that satisfy positive associations \cite{gkr88, alex96}. The fact that $ \{f + \ell \ge 0\} $ has an infinite connected component at high levels $\ell \gg 1$ could also be deduced using the arguments of \cite{molchanov1983percolationii,bm17} (these works relate to the slightly different setting of non-negative shot noise fields, although the arguments go through). Hence the main novelty of Theorems \ref{t:main1} and \ref{t:main2} are (i) the statements that~$ \{f + \ell \ge 0\}$ (resp.\ $\{f^\eta \ge 0\}$) has an infinite connected component for every $\ell > 0$ (resp.\ $\eta > 0$), and (ii) the quantitative estimates for~$\{f \ge 0\}$. 

\smallskip
Our proof of Theorem \ref{t:main1} is largely inspired by \cite{mv18}, which proved an analogous result for planar Gaussian fields, following \cite{bg17,rv17b}. Nevertheless, since the marginals of shot noise fields are only accessible through their characteristic functions, and in general no expression is available for their density, many of the Gaussian techniques fail in the shot noise setting. Moreover, our emphasis on the role played by log-concavity of the mark distribution is novel (although, of course, the Gaussian distribution is log-concave), and deducing the phase transition at this level of generality requires more care. The phase transition in Theorem \ref{t:main2} (with respect to the intensity) is similar to that established for Poisson-Boolean percolation \cite{att18}, although we use different techniques.

\smallskip
Adapting other arguments in \cite{mv18}, one could extract more precise information on the phase transition than is contained in Theorem \ref{t:main1}. First, one could prove that the phase transition is \textit{sharp}, meaning that in the subcritical phase $\ell < 0$ the excursion set $ \{f + \ell \ge 0\} $ crosses large domains with exponentially small probability (see Remark \ref{r:sharp}). Second, one could prove that the `near-critical window' (see, e.g., \cite[Theorem 1.15]{mv18}) is polynomially small in the scale. For brevity we have not included these results.

\subsection{Optimality of the assumptions}
\label{s:optimal}
While the assumptions in Theorem \ref{t:main1} are more-or-less optimal for the proof, we do not expect them to be necessary for its conclusions, and examining the phase transition under more general conditions is an interesting direction for future work (see Question \ref{q:2} below).

\smallskip
Nevertheless, at least two assumptions are essential: (i) $\mu$ is symmetric, and (ii) $g$ has unbounded support. Indeed, if $\mu$ is not symmetric then the shot noise field is not \textit{self-dual} at the zero level, and so there is no reason to expect that $\ell_c = 0$ in general. On the other hand, it is likely the conclusions of Theorem \ref{t:main1} are true with respect to the critical level~$\ell_c$ (which depends on the field). Moreover, if $g$ has bounded support then the statements in Theorem~\ref{t:main1} about $ \{f \ge 0\} $ may be false. Indeed, the theory of Poisson-Boolean percolation (see \cite[Chapter 8]{bo06} or \cite[Chapter 3]{mr96}) demonstrates that $ \{f \ge 0\} $ contains an unbounded connected component whenever the support of $g$ is contained within a ball of sufficiently small radius. On the other hand, it is plausible that the conclusions of Theorem~\ref{t:main1} are true as soon as the support of $g$ contains a sufficiently large compact set.

\smallskip 
As for the assumption of log-concavity \eqref{a:lc}, while it plays an essential role in the proof of Theorem \ref{t:main1} (although we could also formulate a weaker version involving the boundedness of the \textit{Mills ratio}; see Section \ref{s:mills}) we do not believe it to be essential for the result. We use log-concavity to compare the effect of adding a constant to the mark distribution to the effect of \textit{resampling} this distribution; see Proposition~\ref{p:di}. 

\smallskip
Similarly, while the assumptions that the kernel $g$ is positive, symmetric, and decays rapidly are also essential to the proof, we do not believe them to be necessary for the result. We use $g \ge 0$ to guarantee that the shot noise field is increasing with respect to the marks, which among other things implies that the field is positively associated; see Proposition \ref{p:pa}. The symmetry of $g$ allows us to deduce box-crossing estimates at the zero level; see Proposition~\ref{p:sc} and Theorem~\ref{t:rsw}. Finally, the rapid decay of $g$ is used to compare $f$ to an approximation that is finite-range dependent (see Section \ref{ss:pert}). We require stronger decay in Theorem~\ref{t:main2} compared to Theorem \ref{t:main1} because `sprinkling' arguments are not available (c.f.\ Lemmas \ref{l:sprink} and \ref{l:sprink2}).

\smallskip
The remainder of the assumptions are mostly imposed for technical reasons.

\begin{question}
\label{q:2}
Do the conclusions of Theorems \ref{t:main1} and \ref{t:main2} hold for a continuous symmetric shot noise field \eqref{e:ssn} assuming only that the kernel $g$ has unbounded support?  What if the support of $g$ merely contains a sufficiently large compact set?
\end{question}

\subsection{Outline of the rest of the paper}
In Section \ref{s:prelim} we collect preliminary results, including a study of approximations of the shot noise field that have desirable properties. In Section~\ref{s:zero} we establish the `critical' nature of the zero level $\{f \ge 0\}$, including the quantitative estimates in Theorem~\ref{t:main1}. In Section \ref{s:c1} we complete the proof of the phase transition with respect to the level in Theorem~\ref{t:main1}. In Section \ref{s:c2} we show how the arguments in Section \ref{c:1} can be adapted to deduce the phase transition with respect to the intensity in Theorem~\ref{t:main2}. In Appendix \ref{s:a} we collect some technical results about shot noise fields, including concentration inequalities and verifying \eqref{a:bd} in concrete examples.

\medskip
\section{Preliminary results}
\label{s:prelim}

In this section we collect preliminary results on the shot noise field $f$, and various approximations of this field.

\subsection{Log-concavity and bounded Mills ratio}
\label{s:mills}

We first note that the log-concavity assumption \eqref{a:lc} implies that the \textit{Mills ratio} of the mark distribution $\mu$ is bounded; in fact this weaker condition could replace \eqref{a:lc} in all our results.

\smallskip
Recall that, for a continuous random variable, the \textit{Mills ratio} is defined as the ratio of the survival function to the density. Similarly, for a symmetric distribution $\mu$ define 
\begin{equation}
\label{e:mills}
 c_\text{Mills}(\mu) =  \esssup_{x \ge 0} \frac{\bar F_\mu(x)}{\mu(x)} \in [0, \infty] ,
 \end{equation}
where $\mu(x)$ is the density of $\mu$, $\bar F_\mu(x) =  \int_{s \ge x} d \mu(s)$ is the survival function, and where we use the conventions $1/0 = \infty$ and $0/0 = 0$. In particular, these conventions ensure that $c_\text{Mills}(\mu)$ is finite if $\mu$ has a density which is bounded away from zero on its support.

\begin{lemma} 
\label{l:mr}
If $\mu$ is symmetric and log-concave then $c_\text{Mills}(\mu) < \infty$.
\end{lemma}
\begin{proof}
Since log-concavity is preserved under multiplication and marginalisation, the survival function $\bar{F}_\mu(x) = \int \id_{s \ge x} d\mu(s)$ is log-concave, and hence
\[ \frac{-1}{(\log \bar{F}_\mu(x))'} = \frac{\bar{F}_\mu(x)}{\mu(x)} \]
is non-increasing. Since the symmetry and log-concavity of $\mu$ imply that $\bar F_\mu(0) = 1/2$ and $\mu(0) > 0$, we deduce that $\bar F_\mu(x)/\mu(x)$ is bounded over $x \ge 0$.
\end{proof}

We also observe that if $\mu$ has bounded Mills ratio it has a sub-exponential tail:

\begin{lemma}
\label{l:ed}
If $\mu$ is symmetric and $c_\text{Mills}(\mu) < \infty$ then, for all $x \ge 0$,
\[ \bar F_\mu(x) \le F_\mu(0) e^{- x/c_\text{Mills}(\mu)} . \]
\end{lemma}
\begin{proof}
It is sufficient to prove the result for $x \ge 0$ in the support of $\mu$, for which
\begin{equation*}
\log \bar F_\mu(x)   =   \log \bar F_\mu(0)  -  \int_{0}^x  \frac{\mu(s)}{  \bar F_\mu(s) }  ds   \le \log \bar F_\mu(0) -  x/c_\text{Mills}(\mu) . \qedhere
\end{equation*}
\end{proof}

\subsection{Approximations of the shot noise field}
\label{ss:pert}
We next introduce three approximations of the shot noise field $f$ that have desirable properties, based respectively on (i) truncating the kernel, (ii) spatial discretisation, and (iii) adding a constant to the mark. 

\smallskip
The first will be used to compare the field to one which is finite-range dependent. The second will be used, together with the first, to compare the field on compact sets to one which is measurable with respect to a finite set of independent random variables. The third will be used as a proxy for raising the mean of the field.

\subsubsection{Truncation of the kernel}
We first introduce approximations $(f_r)_{r > 0}$ of $f$ that are $r$-range dependent, meaning that $f_r(x)$ and $f_r(y)$ are independent for $|x-y| \ge r$. Fix an infinitely differentiable cut-off function $\chi : [0, \infty) \to [0, 1]$ such that $\chi(x) = 1$ for $x \le 1/4$, and $\chi(x) = 0$ for $x \ge 1/2$. For $r >0$ define the truncated kernel
\begin{equation}
\label{e:gr}
g_r(x) =  g(x) \chi(|x|/r) 
\end{equation}
whose support is contained in the ball $B(r/2)$, and the $r$-truncated shot noise field
\[ f_r =   \sum_{i \in \mathcal{P}} Y_i g_r(x-i) .\]
Clearly $f_r$ is $r$-range dependent as claimed, since $f_r(x)$ depends on the Poisson point process $\mathcal{P}$ restricted to the ball of radius $r/2$ centred at $x$.

\smallskip
We can identify $f$ with the limiting case $r = \infty$ since $f_r$ converges, as $r \to \infty$, to $f$ in probability in the uniform topology on compact sets (we quantify the speed of convergence in Lemma \ref{l:trunc} below).

\subsubsection{Spatial discretisation}
Our second approximation is based on a spatial approximation of the Poisson point process $\mathcal{P}$, and allows us to compare $f$ to a field $f_r^\eps$ that is measurable on compact sets with respect to a finite set of independent random variables. Unlike for $f_r$, we do not define $f_r^\eps$ via an explicit coupling to $f$.

\smallskip
For $\eps > 0$, let $(\textrm{Ber}^{\eps^2}_i)_{i \in \eps \Z^2}$ be i.i.d.\ Bernoulli random variables with parameter $\eps^2$, and define the point set
 \begin{equation}
 \label{e:peps}
 \mathcal{P}^\eps =  \{ i \in \eps \Z^2 : \textrm{Ber}^{\eps^2}_i = 1 \} . 
 \end{equation}
 Then define the $\eps$-discretised shot noise field 
\[  f^\eps(x) =  \sum_{i \in \eps \Z^2} \textrm{Ber}^{\eps^2}_i Y_i g(x-i)  = \sum_{i \in \mathcal{P}^\eps}  Y_i g(x-i) \]
where $(Y_i)_{i \in \eps \Z^2}$ are, as before, i.i.d.\ random variables drawn from $\mu$. Since $g$ is in $C^2$ and $Y_i$ has finite mean, this series converges absolutely in the Sobolev space $W^{2,1}$ on compact sets. For $r > 0$ define $f_r^\eps$ analogously by replacing $g$ with $g_r$. Then for a compact set $D \subset \mathbb{R}$, $f_r^\eps$ is measurable with respect to the finite set of independent random variables $(\textrm{Ber}^{\eps^2}_i, Y_i)_{i \in \Lambda}$ where $\Lambda \subset \eps \Z^2$ is the set of all points within distance $r/2$ of $D$.

\smallskip
We verify in Lemma \ref{l:dis} below that $f_r^\epsilon$ converges, as $\epsilon \to 0$, to $f_r$ in law in the uniform topology on compact sets. Hence it is natural to include $f_r$ in the set $(f_r^\epsilon)_{\epsilon \ge 0}$ by identifying $f_r$ with the limiting case $f_r^0$.
 
\subsubsection{Adding a constant to the mark}
Finally we introduce an approximation of $f$ in which a constant is added to each mark; we use this as a proxy for raising the mean of the field.

\smallskip
Let $h = (h_i)_{i \in \epsilon \mathbb{Z}^2}$, $h_i \in \R$. Then for $\eps > 0$ we construct $f^{\epsilon,h}$ from $f^\epsilon$ by adding $h_i$ to the mark $Y_i$, i.e.\ 
\[ f^{\epsilon,h} = \sum_{i \in \eps \Z^2} \textrm{Ber}^{\eps^2}_i (Y_i + h_i) g(x-i) . \]
In the case $\eps = 0$, we instead let $h = (h_i)_{i \in \mathcal{P}}$ and define 
\[ f^h =  f^{0, h} = \sum_{i \in \mathcal{P}}  (Y_i + h_i) g(x-i) . \]
We also define $f_r^{\epsilon,h}$ analogously (i.e.\ replacing $g$ with $g_r$). When $h \in \mathbb{R}$, we define $f_r^{\eps, h}$ by  setting $h_i \equiv h$. 

\smallskip
If we assume $g \ge 0$, the field $f_r^{\eps,h}$ is increasing in $h_i$, which will be important in the proof of Proposition \ref{p:di} below.

\subsection{Analysis of the approximations}
\label{ss:an}

We next derive bounds on the distance between the shot noise field $f$ and the approximations introduced above. For this we use general concentration inequalities for shot noise fields which, since they are rather technical, we state and prove in Appendix \ref{s:a}.

\smallskip
We begin by analysing the effect of the truncation. Recall that $B(r)$ denotes the ball of radius $r$ centred at the origin. For a function $s:\mathbb{R}^d \to \mathbb{R}$, $d\ge 1$, we use $\|s\|_p$ to denote the standard $L^p$ norm, and for a compact domain $D \subset \R^2$ we define $\|s\|_{D, \infty} = \sup_{x \in D} |s(x)|$. Rather than work under Condition \ref{c:1}, in this section we prefer to isolate the conditions that are relevant:
\begin{condition}
\label{a:c2}
The mark distribution $\mu$ has a sub-exponential tail. Moreover, the kernel $g$ is in $C^2(\R^2)$ and there exist $\alpha > 2$ and $c > 0$ such that
\[  |\partial^{k} g(x)| < c(1+|x|)^{-\alpha-|k|}  \]
for $|k| \le 2$ and $x \in \R^2$.
\end{condition}

\begin{lemma}
\label{l:trunc}
Suppose Condition \ref{a:c2} holds for a parameter $\alpha > 2$. Then:
\begin{enumerate}
\item There exist $c_1, c_2 > 0$ such that, for all $r,s,t \ge 6$,
\[ \mathbb{P}(\|f - f_r\|_{B(s), \infty} > c_1 t^2 r^{2-\alpha} (\log r)^2) < c_1 s^2 e^{-c_2 t} .\]
\item If moreover $\mu$ is symmetric, there exist $c_1, c_2 > 0$ such that, for all $r,s,t \ge 6$,
\[ \mathbb{P}(\|f - f_r\|_{B(s), \infty} > c_1 t^2 r^{1-\alpha} (\log r)^2) < c_1 s^2 e^{-c_2 t} .\]
\end{enumerate}
\end{lemma}

\begin{proof}
Observe that $f^r = f-f_r$ is a shot noise field with mark distribution $\mu$ and kernel 
\[ g^r(x) = g(x) - g_r(x) = g(x)(1 - \chi(|x|/r)) \]
where $\chi$ is, as in \eqref{e:gr}, an infinitely differentiable cut-off function such that $\chi(x) = 1$ for $x \le 1/4$, and $\chi(x) = 0$ for $x \ge 1/2$. To control the deviations of $g^r$ we apply the concentration inequalities in Proposition \ref{p:ldb2}, noting that we may take $\beta =1$ in this proposition since $\mu$ has a sub-exponential tail.

To prove the first statement define the auxiliary function
\[ \widehat{g^r}(x) = \sup_{y\in [-1/2,1/2]^{2}} (1+\log(|x+y|))^2 |g^r(x+y)| .\]
According to \eqref{eq:concentr-sup} (recall that we take $\beta = 1$), for all $r,s,t \ge 6$,
\[ \mathbb{P} \Big(\|f -f_r\|_{B(s), \infty} \ge   t  \big(\|\widehat{g^r}\|_1+ \sqrt{2t} \|\widehat{g^r}\|_2 + \frac{t}{3}  \|\widehat{g^r}\|_{\infty}  \big) \Big) \le  c_1 s^2 e^{-c_2 t}  \]
for some $c_1,c_2 > 0$. By the decay of $g$ in Condition \ref{a:c2}, there is a $c_3 > 0$ such that, for $r \ge 6$,
\[ |g^r(x)| \le c_3 |x|^{-\alpha} \id_{|x| \ge r/4}  \quad \text{and} \quad  |\widehat{g^r}(x)|  \le c_3 |x|^{-\alpha} (\log |x|)^2 \id_{|x| \ge r/4} . \]
Hence there exists a $c_4 > 0$ such that, for $r \ge 6$,
\[ \|\widehat{g^r}\|_1 \le  c_4 r^{2-\alpha} (\log r)^2    \ , \quad  \|\widehat{g^r}\|_2 \le c_4 r^{1-\alpha} (\log r)^2  \quad \text{and} \quad   \|\widehat{g^r}\|_\infty \le c_4 r^{-\alpha} ( \log r)^2  , \]
and so  the bound reduces to 
\[ \mathbb{P} \big(\|f-f_r\|_{B(s), \infty} \ge  c_5 t^{2}  r^{2-\alpha} (\log r)^2  \big) \le  c_1 s^2 e^{-c_2 t} \]
for some $c_5 > 0$.

For the second statement, define the auxiliary functions
\[  \widetilde{g^r}(x) =  (1+\log(|x|))^2 |g^r(x)|  \quad \text{and} \quad \widehat{\nabla g^r}(x) = \sup_{y\in [-1/2,1/2]^{2}} (1+\log(|x+y|))^2 |\nabla g^r(x+y)|   .\]
According to \eqref{eq:concentr-grad} (again recall that we take $\beta = 1$), for all $r,s,t \ge 6$,
\begin{align*}
& \mathbb{P} \Big(\|f -f_r\|_{B(s), \infty} \ge  2 {t}  \big(\sqrt{2t} \|\widetilde{g^r}\|_2+\frac{t}{3}\|\widetilde{g^r}\|_\infty + \|\widehat{\nabla g^r}\|_1 + \sqrt{2t} \|\widehat{\nabla g^r}\|_2  + \frac{t}{3}  \|\widehat{\nabla g^r}\|_\infty  \big) \Big)  \\
& \qquad \le  c_1 s^2 e^{-c_2 t} 
\end{align*}
for some $c_1,c_2 > 0$.  By the decay of $g$ and its derivatives in Condition \ref{a:c2}, and since the derivatives of $\chi(|x|/r)$ are uniformly bounded for $r \ge 1$, there is a $c_3 > 0$ such that, for $r \ge 6$,
\[ |\widetilde{g^r}(x)| \le c_3 |x|^{-\alpha} (\log |x|)^2  \id_{|x| \ge r/4}  \quad \text{and} \quad   |\widehat{\nabla g^r}(x)|  \le c_3 |x|^{-\alpha-1} (\log |x|)^2 \id_{|x| \ge r/4}  .\]
Hence there exists a $c_4 > 0$ such that, for $r \ge 6$,
\[ \|\widetilde{g^r}\|_2 \le c_4 r^{1-\alpha}(\log r)^2 \ , \quad   \|\widetilde{g^r}\|_\infty \le c_4 r^{-\alpha}(\log r)^2  \ ,    \]
\[ \|\widehat{\nabla g^r}\|_1 \le c_4  r^{1-\alpha} (\log r)^2 \ , \quad \|\widehat{\nabla g^r}\|_2 \le c_4 r^{ -\alpha} (\log r)^2  \quad \text{and} \quad   \|\widehat{\nabla g^r}\|_\infty \le c_4 r^{-1-\alpha} ( \log r)^2  , \]
and so the bound reduces to
\[ \mathbb{P} \big(\|f-f_r\|_{B(s), \infty} \ge  c_5 t^{2}  r^{1-\alpha} (\log r)^2  \big) \le  c_1 s^2 e^{-c_2 t} \]
for some $c_5 > 0$.
\end{proof}

We next analyse the effect of adding a constant to the mark:  

\begin{lemma}
\label{l:h}
Suppose Condition \ref{a:c2} holds for a parameter $\alpha > 2$. Then there exist $c_1, c_2 > 0$ such that, for all $h > 0$ and $s,t \ge 1$,
\[ \mathbb{P}(\|f - f^h\|_{B(s), \infty} > c_1 t h ) < c_1 s^2 e^{-c_2 t} .\]
\end{lemma}
\begin{proof}
Observe that $(f - f^h)/h$ is a shot noise field with kernel $g$ and (degenerate) mark distribution $\mu'$ with an atom at $1$.  Define the auxiliary function
\[  \widehat{g}(x) = \sup_{y\in [-1/2,1/2]^{2}} |g(x+y)| . \]
 According to the concentration inequality in equation \eqref{eq:concentr-sup} of Proposition \ref{p:ldb2} (we may take $\beta = \infty$ in this proposition since the support of $\mu'$ is contained in $[-1,1]$), for all $h > 0$ and $s,t \ge 1$,
\[ \mathbb{P} \Big(\|f -f^h\|_{B(s), \infty} \ge  2 h \big( \|\widehat{g}\|_1 + \sqrt{2t} \|\widehat{g}\|_2   + \frac{t}{3}  \|\widehat{g}\|_\infty  \big) \Big) \le  c_1 s^2 e^{-c_2 t} \]
for some $c_1, c_2 > 0$. By the decay of $g$ in Condition \ref{a:c2}, each of $\|\widehat{g}\|_1$, $\|\widehat{g}\|_2$ and $\|\widehat{g}\|_\infty $ are finite, so the bound reduces to
\[ \mathbb{P} \Big(\|f -f^h\|_{B(s), \infty} \ge  c_3 h t  \Big) \le  c_1 s^2 e^{-c_2 t} \]
for some $c_3 > 0$.
\end{proof}

Finally we consider the spatial discretisation, for which we need only the qualitative fact that $f^\eps$ converges to $f$ in law in the uniform topology on compact sets:

\begin{lemma}
\label{l:dis}
Suppose Condition \ref{a:c2} holds for a parameter $\alpha > 2$. Then $f^\eps \Rightarrow f$ in law in the uniform topology on compact sets.
\end{lemma}

\begin{proof}
Fix a compact set $D \subset \R^2$. Recall the point set  $\mathcal{P}^\eps =  \{ i \in \eps \Z^2 : \textrm{Ber}^{\eps^2}_i = 1 \} $ defined in \eqref{e:peps}. Fix $c > 0$ and define the point sets $\mathcal{P}_c = \mathcal{P} \cap [-c,c]^2$ and $\mathcal{P}^\eps_c = \mathcal{P}^\eps \cap [-c,c]^2$. Then
\[  \tilde{f} = \! \! \sum_{i \in \mathcal{P}_c} Y_i^{\mu} g(x-i) \quad \text{and} \quad  \tilde{f}^\eps = \! \! \sum_{i \in \mathcal{P}^\eps_c} Y_i^{\mu}  g(x-i)  \]
are the contributions to $f$ (resp.\ $f^\eps$) from the Poisson points (resp.\ Bernoulli points) inside~$[-c,c]^2$.  Since $f$ and $f^\eps$ are defined as absolutely convergence series in the Sobolev space $W^{2,1}$ on compact sets, the differences
\[  \|  f -   \tilde{f} \|_{D,\infty}  \quad \text{and} \quad  \|f^\eps -  \tilde{f}^\eps\|_{D,\infty}  \]
can be made arbitrarily small with arbitrarily high probability by taking $c > 0$ large enough. Hence it is sufficient to show that one can couple the (ordered) point sets $\mathcal{P}_c$ and $\mathcal{P}^\eps_c$ so that, as $\eps \to 0$, $\| \tilde{f} -  \tilde{f}^\eps    \|_{D, \infty} \to 0$ in probability.

Let $(p_i)_{i \ge 1}$ and $(q_i)_{i \ge 1}$ be enumerations of $\mathcal{P}_c$ and $\mathcal{P}^\eps_c$ respectively. By the standard binomial approximation of a Poisson point process, one can couple $(p_i)$ and $(q_i)$ so that the event $C_\eps = \{ |\mathcal{P}_c| = |\mathcal{P}^\eps_c|\}$ occurs with probability tending to one as $\eps \to 0$, and further, on the event $C_\eps$, $d_0 = \sum_{i} | p_i - q_i| \to 0$ in probability as $\eps \to 0$. Then since $|\mathcal{P}_c|$ and $|Y_i|$ are bounded in probability, on the event $C_\eps$, 
\begin{align*}
  \|\tilde{f} -  \tilde{f}^\eps \|_{D, \infty}  & =   \sup_{x \in D} \Big| \sum_i |Y_i^{\mu}| ( g(x-p_i)  - g(x-q_i) )  \Big| \\
  & \le d_0 \| |\nabla g| \|_\infty   \max_i  |Y_i|    
  \end{align*}
 tends to $0$ in probability as $\eps \to 0$, as required.
 \end{proof}

\subsection{Regularity of critical points}
To complete the section we prove that the critical points of $f$ are unlikely to be close to any fixed level. 

\begin{proposition}[Critical points in a narrow window]
\label{p:regcrit}
Suppose Condition \ref{a:c2} holds for a parameter $\alpha > 2$ and also \eqref{a:bd} holds. Then there exist $c, \delta_0 > 0$ such that, for every $\ell \in \mathbb{R}$, $\delta \in (0, \delta_0)$ and $s \ge 1$,  
\[ \mathbb{P}( \exists  x \in B(s) : \nabla f(x) = 0, |f(x)-\ell| < \delta) < c s^2 \delta  |\!\log \delta|^2  , \]
and moreover, for every direction $v$ on the unit circle, 
\[ \mathbb{P}( \exists  x \in s L_v : \partial_v f(x) = 0, |f(x)-\ell| < \delta) < c s \delta  |\!\log \delta|  , \]
where $L_v$ denotes a unit line segment in direction $v$.
\end{proposition}

\begin{corollary}[Regularity of level lines]
\label{c:rll}
Suppose Condition \ref{a:c2} holds for a parameter $\alpha > 2$ and also \eqref{a:bd} holds. Then for every $\ell \in \mathbb{R}$, the level set $ \{ f = \ell\}$ almost surely consists of a collection of simple closed $C^2$-smooth curves. Moreover, for every line segment~$L$, $\{ f = \ell\}$ almost surely has no points of tangency with $L$.\end{corollary}
\begin{proof}
Taking $\delta \to 0$ and then $s \to \infty$ in Proposition \ref{p:regcrit}, we deduce that almost surely $f$ has no critical points at level $\ell$. Then since $f$ is $C^2$-smooth, the components of $ \{ f = \ell\}$ are simple closed $C^2$-smooth curves by the implicit function theorem. Similarly, the restriction of $f$ to $L$ almost surely has no critical points at level $\ell$, which implies that $ \{ f = \ell\}$ has no points of tangency with $L$.
\end{proof}

\begin{proof}[Proof of Proposition \ref{p:regcrit}]
For each multi-index $k$ such that $|k| \le 2$, the random field $\partial^k f$ is a symmetric shot noise field with kernel $\partial^k g$ and mark distribution $\mu$. Define the auxiliary functions
\[  \widehat{\partial^k g}(x) = \sup_{y\in [-1/2,1/2]^{2}} (1+\log(|x+y|))^2 |\partial^k g(x+y)| .\]
Recalling that 
\[  |\partial^{k} g(x)| < c(1+|x|)^{-\alpha-|k|}  \]
for all $|k| \le 2$, each of $\|\widehat{\partial^k g}\|_1$, $\|\widehat{\partial^k g}\|_2$ and $\|\widehat{\partial^k g}\|_\infty$ are finite. Applying equation \eqref{eq:concentr-sup} of Proposition \ref{p:ldb2} (we may take $\beta =1$ in this proposition since $\mu$ has a sub-exponential tail), there exist constants $c_1, c_2 > 0$ such that, for all $|k| \le 2$ and $t \ge 1$,
\begin{equation}
\label{e:regcrit1}
\mathbb{P}(\|\partial^k f\|_{[0,1]^{2}, \infty} \ge c_1 t^2 ) \le c_1 e^{- c_2 t} .
\end{equation}
Now define $t_\delta = |\log \delta|^{1/2}$, and let $\Omega_{\delta }$ be the event that
\[ \max_{|k| = 1,2}\|\partial^k f\|_{[0,1]^{2},\infty} \le  t_\delta^2  \]
which, by \eqref{e:regcrit1}, has probability greater than $1 - c_1 \exp(- c_3 t_\delta)$ for some $c_3 > 0$ and all sufficiently small $\delta > 0$. Choose an integer $n \ge  t_{\delta}^2/\delta $, and cover $B(s)$ with $6 s^2 n^2$ squares $S_{1},\dots ,S_{6 s ^2 n^2}$ of side-length $1/n$. Let $x_{k}$ be the centre of the square $S_{k}$. Then we have, for some $c_4, c_5 > 0$,
\begin{align*}
& \mathbb{P}(\exists x\in B(s) :\nabla f(x)=0 ,|f(x)-\ell|\le \delta ) \\
& \quad \le \sum_{k=1}^{6 s ^2 n^2} \mathbb{P}(\Omega _{\delta },\;\exists x\in S_{k}:\nabla f(x)=0 , |f(x)-\ell|\le \delta )+\mathbb{P}(\Omega _{\delta }^{c})\\
 & \quad \le \sum_{k=1}^{6 s ^2 n^2}\mathbb{P}( |\nabla f(x_{k})|\le t_\delta^{2} /n , |f(x_{k}) - \ell|\le \delta + t_\delta^{2} /n)+\mathbb{P}(\Omega _{\delta }^{c}) \\ 
 & \quad \le c_4 s^2 n^{2} \times (t_\delta^2/n)^2 \times (\delta + t_\delta^{2}/n) +\mathbb{P}(\Omega _{\delta}^{c})  \\
 & \quad \le  2 c_4 s^2 \delta t_\delta^4 +   c_1 \exp(- c_3 t_{\delta}) \le c_5 s^2 \delta t_\delta^4,
\end{align*}
with the third inequality following from the bound on the density of $(f(0), \nabla f(0))$, and the final inequality holding for all sufficiently small $\delta > 0$. 

The proof of the second statement is analogous, except we choose an $n \ge  t_{\delta}/\delta$ and cover $s L_v$ with $2 s  n$ line segments of length $1/n$.
\end{proof}

\medskip

\section{Criticality of the zero level}
\label{s:zero}

In this section we study the `critical' properties of the excursion set $ \{f \ge 0\}$. In particular, we establish (i) positive association for crossing events, (ii) quasi-independence for crossing events, (iii) `box-crossing' estimates, (iv) one-arm decay estimates, and finally (v) the absence of percolation. Together this completes the proof of the first, third and fourth statements of Theorem \ref{t:main1}. Since many of the arguments are standard in percolation theory, we emphasise the aspects that are specific to the shot noise setting.

\smallskip
Throughout this section we suppose that Condition \ref{c:1} holds for a parameter $\alpha > 3$, and also \eqref{a:bd}, but we stress that the log-concavity assumption \eqref{a:lc} plays no role. Moreover, only the first approximation $f_r$ in Section \ref{ss:pert} will be used in this section (except briefly in the proof of Proposition \ref{p:pa} where we use the discretisation $f^\eps$ as a technical tool).

\subsection{Crossing events}
We begin by introducing `crossing events' for rectangles and annuli. Let $s$ be a continuous planar function (which will later stand for realisations of the fields $f$, $f_r$ etc.) and let $\ell \in \mathbb{R}$ be a level. For $\rho_1,\rho_2 > 0$, define $R[\rho_1,\rho_2] = [0, \rho_1] \times [0, \rho_2]$, and let $\{s \in \cross_\ell(\rho_1,\rho_2)\}$ denote the event that there exists a path in $\{s + \ell \ge 0\} \cap R[\rho_1,\rho_2]$ that intersects both the `left' and `right' sides of $R[\rho_1,\rho_2]$, i.e.\ intersects both $\{0\} \times [0, \rho_2]$ and $\{\rho_1\} \times [0, \rho_2]$. Moreover, let $\{s \in \cross^\ast_\ell(\rho_1,\rho_2)\}$ denote the event that there exists a path in $\{s + \ell \le 0\} \cap R[\rho_1, \rho_2]$ that intersects both the `top' and `bottom' sides of $R[\rho_1, \rho_2]$; we show in Lemma \ref{l:prop} below that, in our setting, $\{s \in \cross^\ast_\ell(\rho_1,\rho_2)\}$ is the complement of $\{s \in \cross_\ell(\rho_1,\rho_2)\}$ up to a null set.

\smallskip
Similarly, for $0 < \rho_1 < \rho_2$ define $A[\rho_1, \rho_2] = \{ [-\rho_2, \rho_2]^2 \setminus [-\rho_1, \rho_1]^2 \}$, and let $\{s \in \arm_\ell(\rho_1, \rho_2)\}$ denote the event that there exists a path in  $\{s + \ell \ge 0 \} \cap A[\rho_1,\rho_2]$ that intersects both $[-\rho_1, \rho_1]^2$ and $\partial [-\rho_2, \rho_2]^2$.

\smallskip
Collectively we shall refer to the events $\cross_\ell$, $\cross^\ast_\ell$ and $\arm_\ell$ as `crossing events'. To each crossing event we will associate its \textit{level} $\ell$, and also its \textit{supporting domain}, being the compact domain $D \subset \mathbb{R}^2$ on which the event is defined (i.e.\ either the rectangle $R[\rho_1, \rho_2]$ or the annulus $A[\rho_1, \rho_2]$). Notice that the crossing events $\{s \in \cross_\ell\}$ and $\{s \in \arm_\ell\}$ are increasing in both the function $s$ and the level $\ell$ (i.e.\ $\{s \in \cross_\ell\}$ implies that $\{s' \in \cross_{\ell'}\}$ for any $s' \ge s$ and $\ell' \ge \ell$), whereas the event $\cross^\ast_\ell$ is decreasing in both the field and the level. Henceforth we will refer to these simply as increasing (resp.\ decreasing) crossing events.

\smallskip
Let us state some basic properties of these events:

\begin{lemma}
\label{l:prop} $\,$
\begin{enumerate}
\item Let $E$ be a crossing event. Then $\{f \in E\}$ is a continuity set in the uniform topology on the supporting domain $D$ of $E$, i.e.\ if $\{f \in E\}$ occurs then almost surely there exists a $\delta > 0$ such that $\{f + s  \in E \}$ for all functions $s$ such that $\|s\|_{D, \infty} < \delta$.  
\item The events
\[ \{ f \in  \cross_\ell(\rho_1,\rho_2)  \} \quad \text{and} \quad \{f \in  \cross^\ast_\ell(\rho_1,\rho_2) \} \]
form a partition of the probability space up to a null set.
\end{enumerate}
\end{lemma}
\begin{proof}
These are simple consequence of the regularity of the level lines in Corollary~\ref{c:rll}. We assume that $\ell = 0$, since the proof is identical for all $\ell \in \mathbb{R}$. 

For the first statement, we prove that $\{f \in \cross_0(\rho_1, \rho_2)\}$ is a continuity event (the proof for the other crossing events is similar). We consider $R[\rho_1, \rho_2]$ to be a set stratified by its four boundary line segments and its four corners. For $\delta > 0$ define the event $\Omega_\delta$ that $f$ has no stratified critical points such that $|f(x)| \le \delta$. More precisely, $\Omega_\delta$ occurs if
\begin{itemize}
\item $f$ has no critical point $x \in R[\rho_1, \rho_2]$ such that $|f(x)| \le \delta$;
\item For each boundary line segment $L \subset \partial R[\rho_1, \rho_2]$, $f|_L$ has no critical point $x \in L$ such that $|f(x)| \le \delta$;
\item None of the corners of $R[\rho_1,\rho_2]$ have $|f(x) | \le \delta$.
\end{itemize}
By the (stratified) Morse lemma \cite[Theorem 7]{han02}, if $\Omega_\delta$ occurs then, restricted to $R[\rho_1, \rho_2]$, the (stratified) diffeomorphism class (i.e.\ with respect to diffeomorphims that fix the boundary sides) of $\{f + t \ge 0\}$ is constant for $|t| \le \delta$. Since $\{f  + t \in \cross_0(\rho_1, \rho_2) \}$ depends only on this diffeomorphism class, we deduce that $\{f - t \in \cross_0(\rho_1, \rho_2) \}  \cap \Omega_\delta$ agree for all $|t| \le \delta$.  To conclude remark that, by Corollary~\ref{c:rll}, there is almost surely a $\delta > 0$ such that $\Omega_\delta$ occurs; since $\cross_0$ is increasing, this yields the continuity in the uniform topology.

For the second statement, we assume $\{ f \in  \cross_0(\rho_1,\rho_2)  \}$ and $\{ f \in  \cross^\ast_0(\rho_1,\rho_2)  \}$ both occur and derive a contradiction (the proof that their complements cannot both occur is similar). By the first statement of the lemma, there is almost surely a $\delta > 0$ such $\{ f - \delta \in  \cross_{0}(\rho_1,\rho_2)  \}$ and $\{ f + \delta \in  \cross^\ast_0(\rho_1,\rho_2)  \}$ occur, which implies that there is a path in $\{f - \delta \ge 0\}$ between the left and right sides of $R[\rho_1, \rho_2]$ and a path in $\{f + \delta \le 0\}$ between the top and bottom sides of $R[\rho_1, \rho_2]$. Since any such paths must intersect, we have a contradiction.
\end{proof}

\begin{remark}
\label{r:trunc1}
Note that the conclusions of Lemma \ref{l:prop} are not true for the truncated fields $f_r$. Indeed, there is a positive probability that the zero-level set $\{f_r = 0\}$ covers $R[\rho_1,\rho_2]$, which means that the crossing events are discontinuous in the uniform topology, and moreover,
\[ \P(f_r \in  \cross_0(\rho_1,\rho_2) ) + \P(f_r \in  \cross^\ast_0(\rho_1,\rho_2) ) > 1 .    \]
\end{remark}

We observe a simple consequence of the `self-duality' of the zero level:

\begin{proposition}[Square crossings]
\label{p:sc}
Let $\rho > 0$. Then
\[ \mathbb{P}(f \in \text{Cross}_0(\rho, \rho) ) = 1/2 . \] 
\end{proposition}
\begin{proof}
By Lemma \ref{l:prop}, $\{ f \in  \cross_0(\rho,\rho)  \}$ and $\{f \in  \cross^\ast_0(\rho,\rho) \}$ partition the probability space. On the other hand, by the symmetry of the mark $\mu$ and the symmetry of the kernel $g$ under reflection in the line $\{y=x\}$, these events have equal probability.
\end{proof}

\subsection{Positive associations}
We next verify the crucial `positive association' property for crossing events (this is one of two places in the proof of Theorem \ref{t:main1} that we use the positivity of the kernel $g \ge 0$, the other being in the proof of Propositions \ref{p:di}):

\begin{proposition}[Positive associations]
\label{p:pa}
Let $E_1$ and $E_2$ be crossing events that are either both increasing or both decreasing. Then
\[ \mathbb{P}( \{f  \in E_1 \} \cap \{f  \in E_2 \} )\ge \mathbb{P}(f  \in E_1 ) \mathbb{P}(f  \in E_2) . \] 
\end{proposition}
\begin{proof}
Without loss of generality suppose the events are both increasing. Recall the spatially discretised field $f^\eps$, and for $i \in \eps \Z^2$ define $Z_i = \textrm{Ber}^{\varepsilon^2}_i Y_i$ so that
\[ f^\eps(x) = \sum_{i \in \eps \Z^2} Z_i g(i-x)  .   \]
Since $g \ge 0$, $\id_{\{f^\eps  \in E_1 \}}$ and $\id_{\{f^\eps  \in E_2 \}}$ can be viewed as increasing functions of the i.i.d.\ sequence $(Z_i)_{i \in \eps \Z^2}$. Hence by the classical Harris/FKG inequality for product measures (see \cite[Section 2.2]{gr99} for the case of Bernoulli random variables, and the proof is identical in the general case),
 \[ \mathbb{P}(\{f^\eps  \in E_1 \} \cap \{f^\eps  \in E_2 \}) \ge \mathbb{P}(f^\eps  \in E_1 ) \mathbb{P}(f^\eps \in E_2 ) . \] 
Since $f^\eps \Rightarrow f$ in law in the uniform topology on compact sets (Lemma \ref{l:dis}), and crossing events are continuity events for $f$ in this topology (Lemma \ref{l:prop}), we conclude by taking $\eps \to 0$.
\end{proof}

\begin{remark}
\label{r:trunc2}
It is natural to expect that the conclusion of Proposition \ref{p:pa} also holds for the truncated field~$f_r$, however the proof does not go through because crossing events are not continuity events for $f_r$ (see Remark \ref{r:trunc1}).
\end{remark}

\subsection{Quasi-independence}
We next show that crossing events that are supported on well-separated domains are approximately independent. 

\smallskip
As a preliminary, we state a general result that bounds the effect of truncation on the probability of crossing events. Recall the constant $\alpha > 3$ from Condition \ref{c:1}.

\begin{proposition}
\label{p:perturbation}
There exist $c_1, c_2> 0$ such that, for every $r,R,t \ge 6$, every compact set $D \subset \mathbb{R}^2$ of diameter at most $R$, and every collection $(E_i)_{i \le n}$ of crossing events whose supports are contained in $D$,
\begin{equation}
\label{e:pqi}
\left| \P \left(f \in E   \right) - \P \left( f_r \in E  \right)  \right| < c_1 n R^2  t^2 r^{1-\alpha} (\log r)^2  \big( (\log r)^2 + (\log t)^2  \big) +  c_1 R^2 e^{-c_2 t}  ,
\end{equation}
where $E = \cap_{i \le n} E_i$.
\end{proposition}

\begin{proof} 
For each crossing event $E_i$, let $\ell_i$ and $S_i \subset D$ denote its associated level and support respectively. Recall that each $S_i$ has a piece-wise linear boundary (with one or two connected components, depending on the crossing event). For $\delta > 0$, define the event $\Omega^1_\delta$ that, for each~$S_i$:
\begin{itemize}
\item $f$ has no critical point $x \in S_i$ such that $|f(x) - \ell_i| \le \delta$;
\item For each boundary line segment $L \subset \partial S_i$, $f|_L$ has no critical point $x \in L$ such that $|f(x) - \ell_i| \le \delta$;
\item None of the corners of $S_i$ have $|f(x) - \ell_i| \le \delta$.
\end{itemize}
Applying the Morse lemma as in the proof of Lemma \ref{l:prop}, we deduce that the events
\[ \{f \in E_i \} \cap \Omega^1_\delta  \, , \quad \{f + c \in E_i \}\cap \Omega^1_\delta   \quad \text{and} \quad \{f - c \in E_i \}  \cap \Omega^1_\delta   \]
agree for all $i$ and all $c \in [-\delta,\delta]$. Define now the event $\Omega^2_\delta$ that
\[ \| f - f_r \|_{D,\infty} \le \delta . \]
Since crossing events are either increasing or decreasing, on $\Omega^1_\delta \cap \Omega^2_\delta$ the events 
\[ \{f \in E_i\} \quad \text{and} \quad  \{f_r \in E_i \}  \]
agree, and hence so do $\{f \in E\}$ and $\{f_r \in E\}$. To finish the proof, recall that by the second item of Lemma \ref{l:trunc} there exist $c_1, c_2 > 0$ such that $\P(\Omega^2_\delta) > 1 - c_1 R^2 e^{-c_2 t}$ for the choice
\begin{equation}
\label{e:c}
\delta =  c_1   t^{2} \, r^{1-\alpha}  (\log r)^2.
\end{equation}
 Moreover, by Proposition \ref{p:regcrit} and the union bound, there is a $c_3 > 0$ such that, for all small enough $\delta > 0$,
\[    \mathbb{P} (\Omega^1_\delta ) > 1 - c_3 n R^2 \delta |\log \delta|^2  . \]
Hence, setting $\delta$ as \eqref{e:c}, there is a $c_4 > 0$ such that
\[    \mathbb{P} ( \Omega^1_\delta \cap \Omega^2_\delta ) > 1 - c_4 n R^2   t^{2} r^{1-\alpha} (\log r)^2 ((\log r)^2 + (\log t)^2 )  - c_1 R^2 e^{-c_2 t}  \]
which gives the result.
\end{proof}

Our main quasi-independence result is a corollary of the previous proposition.

\begin{theorem}[Quasi-independence]
\label{t:qi}
There exist $c_0,c_1, c_2 > 0$ such that, for every $r,R,t \ge c_0$, every pair of compact sets $D_1 \subset \mathbb{R}^2$ (resp.\ $D_2$) of diameter at most $R$ and such that $r = \text{dist}(D_1, D_2)$, and every pair of collections of $n_1$ (resp. $n_2$) crossing events $(E^1_i)_{i \le n_1}$ (resp.\ $(E^2_i)_{i \le n_2})$ that are supported on $D_1$ (resp.\ $D_2$),
\begin{multline}
\label{e:qi}
\qquad \qquad \left| \P \left(  f \in E^1 \cap E^2  \right) - \P \left( f \in E^1 \right) \P \left( f \in E^2 \right) \right| \\
 < c_1 \max\{n_1,n_2\} R^2   t^{2} r^{1-\alpha}  (\log r)^2 ((\log r)^2 + (\log t)^2 ) + c_1  R^2 e^{-c_2 t} 
\end{multline}
where $E^j = \cap_{i \le n_j} E^j_i$ for $j \in \{1,2\}$.
\end{theorem}

\begin{proof}
This follows immediately from Proposition \ref{p:perturbation} by replacing $f$ with its truncated version $f_r$, and noticing that the events $\{f_r \in E_1\}$ and $\{f_r \in E_2\}$ are independent since they depend on domains that are separated by distance $r$ and $f_r$ is an $r$-dependent field (see the proof of \cite[Theorem 4.2]{mv18} for details).
\end{proof}

\begin{remark}
\label{r:qi}
Equation~\eqref{e:qi} (with the setting $t = (\log R)^2)$ implies that two crossing events that are supported on domains of diameter at most $R$ and separated by a distance greater than $r > c_1 R$ are approximately independent in the sense that, as $R \to \infty$,
\begin{equation}
\label{e.qiorder}
 \left| \P \left(  f \in E_1 \cap E_2  \right) - \P \left(f \in E_1 \right) \P \left( f \in E_2 \right) \right| < c_2 R^{3-\alpha} (\log R)^8  \to 0 .
 \end{equation}
In fact, this is precisely the origin of the assumption $\alpha > 3$ in Condition \ref{c:1}.
\end{remark}

\subsection{Box-crossing estimates}

We next establish the `box-crossing' estimates (also known as Russo-Seymour-Welsh estimates), which state that the probability of crossing events for rectangles are bound away from zero and one uniformly in the scale.

\begin{theorem}[Box-crossing estimates]
\label{t:rsw}
For every $\rho_1, \rho_2 > 0$,
\[     \inf_{R \ge 1} \, \mathbb{P}( f \in \cross_0(R\rho_1, R\rho_2) )    > 0  \quad \text{and} \quad    \sup_{R \ge 1}  \, \mathbb{P} (f \in \cross_0(R\rho_1, R\rho_2) )  < 1. \]
\end{theorem}
\begin{proof}
In \cite{tassion2014crossing}, Tassion showed that the following four conditions, all satisfied in our setting, are sufficient for a translation invariant closed subset of the plane to satisfy the box-crossing estimates (see \cite[Remark 2]{tassion2014crossing}):
\begin{enumerate}
\item The square-crossing property in Proposition \ref{p:sc};
\item The positive association in Proposition \ref{p:pa}; 
\item The symmetry guaranteed by Condition \ref{c:1}; and 
\item The quasi-independence guaranteed by Theorem \ref{t:qi} (and more specifically \eqref{e.qiorder}). \qedhere
\end{enumerate} 
In fact, more recent work (appearing after the first version of this paper) has shown that conditions (1)--(3) are already sufficient \cite{kt20}. Despite this, Theorem \ref{t:qi} is still needed in our work (for instance in the proof of Theorem \ref{t:arm} below, crucial for the main results).
\end{proof}

\subsection{Absence of percolation at the zero level}
Finally, we deduce that the components of $ \{f + \ell \ge 0\} $ are bounded for $\ell \le 0$, completing the proof of the first statement of Theorem~\ref{t:main1}.

\begin{theorem}[One-arm decay and absence of percolation]
\label{t:arm}
There exist $c, c_{\text{Arm}} > 0$ such that, for $1 \leq r \leq R$,
\begin{equation}
\label{e:onearm}
   \mathbb{P} \left(   f  \in \arm_0(r, R) \right)     <  c \left( \frac{r}{R} \right)^{c_{\text{Arm}}}. 
   \end{equation}
In particular, $\{f \ge 0\}$ has bounded components almost surely.
\end{theorem}
\begin{proof}
We follow the proof of \cite[Proposition 4.5]{rv17a}, which shows that \eqref{e:onearm} follows from: 
\begin{enumerate}
\item The positive association in Proposition \ref{p:pa}; 
\item The quasi-independence in Theorem~\ref{t:qi}; and
\item The box-crossing estimates in Theorem \ref{t:rsw}.
\end{enumerate}
Fix $\gamma \in (0,1)$ sufficiently small so such that $2 + (1-\gamma)(1-\alpha) < 0$, possible since $\alpha > 3$. By adjusting the constants $c, c_{\text{Arm}} > 0$ we may assume that $R$ is sufficiently large and that $r \in (R^{1-\gamma}, R/6)$. Let $i_0 =  \lfloor \log_3 ( R / (2r)) \rfloor \ge 1$ and define the crossing events
\[ A(i) =   \arm_0 \big( 3^i r ,  2 \times 3^i  r  \big)  \ , \quad i = 1, \ldots , i_0 .  \]
 Observe that
\[   \arm_0(r, R) \quad \implies \quad \cap_{i = 1, \ldots , i_0} A(i) . \]
Note also that, by Proposition \ref{p:pa} and Theorem \ref{t:rsw} (and standard gluing arguments), there is a $c_1 > 0$ such that $\P( A(i) )   < 1 - c_1$ for all $i = 1, \ldots , i_0$. Then iteratively applying Theorem~\ref{t:qi} (with the setting $t = (\log R)^2$) for $i = 1, \ldots, i_0$, there are $c_2, c_3, \ldots, c_5> 0$ such that
\begin{align*} 
& \P( \arm_0(r, R) )  \le  \P \big( \cap_{i = 1, \ldots , i_0} A(i)   \big)   \\
 & \quad  \le  (1-c_1) \times \P \big( \cap_{i = 2, \ldots , i_0} A(i) \big) +   c_2 i_0 R^2   r^{1-\alpha}   (\log R)^{c_3} + c_2  R^2 e^{-c_4 (\log R)^2}        \\
&  \quad \le  \cdots \\
&  \quad \le (1-c_1)^{i_0} +   \Big(\sum_{i \ge 0} (1-c_1)^i  \Big) c_2 i_0 R^2  r^{1-\alpha}    (\log R)^{c_3} + c_2  R^2 e^{-c_4 (\log R)^2}     \\
& \quad \le (1-c_1)^{i_0} +  c_5 R^{-c_6} 
\end{align*}
for sufficiently large $R$, where in the last inequality we used that $i_0 < (\log R)^2$, $r > R^{1-\gamma}$, and that $2 +(1-\gamma) (1-\alpha) < 0$. Since
\[ (1-c_1)^{i_0} = (1-c_1)^{ \lfloor \log_3 ( R / (2r) ) \rfloor}  \le c_6  (r/R)^{c_7}   \]
for sufficiently large $R$ and some $c_6, c_7 > 0$, we have established~\eqref{e:onearm}.

For the second statement, suppose for contradiction that $\{ f \ge 0\}$ has an unbounded component with positive probability and let $\mathcal{C}$ denote the union of all unbounded components. Since $f$ is $C^2$-smooth, $\mathcal{C}$ has positive Lebesgue measure, and hence by stationarity $\P(0 \in \mathcal{C}) > 0$. This implies that $\{f \in \arm_0(1,R)\}$ occurs with probability bounded below for all $R \ge 1$, which contradicts $\P \left( f \in \arm_0(1,R) \right) \to 0$.
\end{proof}

\begin{remark}
As shown in \cite[Section 4.2]{bg17}, given the quasi-independence in Theorem~\ref{t:qi}, the properties established for the excursion set $ \{f \ge 0\} $ in Theorems \ref{t:rsw} and \ref{t:arm} hold also for the zero level set $\{f = 0\}$.
\end{remark}

\medskip

\section{The phase transition with respect to the level}
\label{s:c1}

We now study the phase transition with respect to the level and complete the proof of Theorem \ref{t:main1}. Throughout this section we assume Condition \ref{c:1}, \eqref{a:bd} and \eqref{a:lc} all hold.

\smallskip
The main intermediate step is to deduce that, for any $\ell > 0$, crossing probabilities for $2R \times R$ rectangles converge to $1$ at a stretched-exponentially rate:

\begin{theorem}
\label{t:sharp}
For every $\ell > 0$ there exist $c_1,c_2 > 0$ such that, for every $R \ge 1$, 
\[  \P \left( f \in \cross_{\ell}(2R,R) \right) \ge 1 -  c_1 e^{-R^{c_2} } . \]
\end{theorem}

Theorem \ref{t:main1} can be deduced from Theorem \ref{t:sharp} (along with Theorems \ref{t:rsw} and \ref{t:arm}) in a straightforward way:

\begin{proof}[Proof of Theorem \ref{t:main1}]
The claim that $\{f \ge 0\}$ has bounded components is proven in Theorem~\ref{t:arm}. The boundedness of the components of $\{f + \ell = 0\}$ is an immediately consequence. Indeed, without loss of generality (and by the symmetry in law of $f$ and $-f$) we may suppose that $\ell \le 0$. Then an unbounded component of $\{f + \ell = 0\}$ would be contained in an unbounded component of $\{f + \ell \ge 0\}$, and hence also in an unbounded component of $\{f \ge 0\}$, which is a contradiction.

The existence of a unique unbounded component of $\{f + \ell \ge 0\}$ for $\ell > 0$ is a consequence of Theorem~\ref{t:sharp} by standard gluing arguments. Indeed, by Theorem \ref{t:sharp} we have
\[ \sum_{k \ge 1} \big(1 - \P \big( \cross_\ell(2^{k+1},2^k) \big) \big) < \infty , \]
and so, by the Borel-Cantelli lemma, there exists a $k_0 \ge 1$ such that
\[  \{ f  \in \cross_\ell(2^{k+1},2^k) \} \]
occurs for each $k \ge k_0$. Arranging the $2^{k+1} \times 2^k$ rectangles (and rotated versions) so that the resulting crossings overlap, this implies the existence of an unbounded connected component in $\{f + \ell \ge 0\}$. For uniqueness, if there were two such unbounded components, there would be an unbounded component of $\{f + \ell < 0\}$ separating them, which would contradict what we proved above.

The remainder of the claims in Theorem \ref{t:main1} are proven in Theorems \ref{t:rsw} and \ref{t:arm}. 
\end{proof}

\begin{remark}
\label{r:sharp}
By bootstrapping the estimate in Theorem \ref{t:sharp} (see \cite[Section 5.1]{kes82} and also \cite{mv18}), one could actually deduce the \textit{exponential convergence} of crossing probabilities
\[  \P \left( f \in \cross_{\ell}(2R,R) \right) \ge 1 -  c_1 e^{-c_2 R}  , \]
which is sometimes referred to as the \textit{sharpness of the phase transition}. Since the arguments are very similar to in \cite{mv18}, we do not prove this here.
\end{remark}

The remainder of the section is devoted to proving Theorem \ref{t:sharp} which completes the proof of Theorem \ref{t:main1}. 

\subsection{A differential inequality}
Recall the truncated discretised approximation $f_r^\eps$ of the field $f$ (see Section \ref{ss:pert}). The crux of the argument is to establish a differential inequality for $\P(f_r^\eps \in \cross_{\ell}(2R,R))$ with respect to the level $\ell$. In fact, it will initially be more natural to differentiate with respect to the mean of the mark distribution, i.e.\ the parameters $h = (h_i)$ in the definition of the field $f^{\epsilon,h}_r$ (see again Section \ref{ss:pert}). Recall the definition of the Mills ratio $c_{\text{Mills}} = c_{\text{Mills}}(\mu)$ in \eqref{e:mills}, which is finite by Lemma \ref{l:mr}. Recall also the (lower-right) Dini derivative of a function $F: \mathbb{R} \to \mathbb{R}$,
\[ \frac{d^-}{dx} F(x) = \liminf_{t \downarrow 0} \frac{F(x+t)-F(x)}{t}  .\]

\begin{proposition}[Differential inequality]
\label{p:di}
There exists a $h_0 > 0$ such that, for all $\ell \in \mathbb{R}$, $r \ge 1$, $R > 2r$, $\epsilon > 0$ and $h \in [0, h_0]$,
\begin{align*}
& \frac{\partial^-}{\partial h}  \Big( - \log \big( 1 - \P( f^{\epsilon,h}_r \in \cross_\ell(2R,R) ) \big)  \Big)  \\
& \qquad \ge \frac{1}{2c_{\text{Mills}}}   \frac{ \P (f^{\epsilon,h}_r \in \cross_{\ell}(2R,R))}{ \inf_{\bar r \in (r, R/2)} \left( 2 \bar r / R + \P( f^{\epsilon}_r \in  \arm_{-\ell}(2r, \bar r) ) \right)  } .
\end{align*}
\end{proposition}

To prove Proposition \ref{p:di} we combine (i) a Russo-type inequality that bounds from below the derivative of $ \P( f^{\epsilon,h}_r \in \cross_\ell(2R,R) )$ in terms of \textit{resampling influences}, and (ii) the OSSS inequality that bounds the resampling influences from below. In the proof we rely crucially on the fact that $g \ge 0$; this is the second place in the proof of Theorem \ref{t:main1} where we use this assumption (along with the proof of Proposition \ref{p:pa}).

\smallskip
Let us first state the Russo-type inequality. Recall that, since the kernel $g_r$ is supported on $B(r/2)$, the event $\{f^{\epsilon,h}_r \in \cross_{\ell}(2R,R)\}$ is measurable with respect to the finite set of independent random variables 
\begin{equation}
\label{e:z}
(Z_i )_{i \in \Lambda}= (\textrm{Ber}^{\eps^2}_i (Y_i + h))_{i \in \Lambda} , 
\end{equation}
where $\Lambda \subset \eps \Z^2$ consists of the $i \in \eps \Z^2$ within distance $r/2$ of $[0, 2R] \times [0, R]$. The \textit{resampling influence} $I_i$ of each coordinate $Z_i$ on $\{f^{\epsilon,h}_r \in \cross_{\ell}(2R,R)\}$ is the probability that resampling $Z_i$ modifies the outcome of $\{f^{\epsilon,h}_r \in \cross_{\ell}(2R,R)\}$. More precisely, if $\mathcal{F}_i$ denotes the $\sigma$-algebra generated by $(Z_j)_{j \in \Lambda \setminus \{i\}}$ and $\tilde{Z}_i$ is an independent copy of $Z_i$,
\begin{equation}
\label{e:i}
  I_i = 2 \E \big[   \P(Z_i  \in \cross_{\ell}(2R,R), \tilde{Z}_i \notin \cross_{\ell}(2R,R) | \mathcal{F}_i )  \big] .
  \end{equation}
 This is well-defined since, conditionally on $\mathcal{F}_i$, the event $\{f^{\epsilon,h}_r  \in \cross_{\ell}(2R,R)\}$ depends only on $Z_i$. 
   
\begin{proposition}[Russo-type inequality]
\label{p:russo}
There exists $h_0 > 0$ such that for all $\ell \in \mathbb{R}$, $r, R, \varepsilon > 0$ and $h \in [0, h_0]$,
\[  \frac{\partial^-}{\partial h}  \P \big( f^{\epsilon,h}_r \in \cross_{\ell}(2R,R) \big)  \ge \frac{1}{4c_{\text{Mills}}}  \sum_{i \in \Lambda} I_i  . \]
\end{proposition}

\begin{proof}
Fix $h_0 > 0$ such that $\P( Y_i \le h_0 )  \le 2  \P ( Y_i \ge  h_0 )$; this is possible since $\mu$ is symmetric and has non-zero density in a neighbourhood of the origin (since its density is log-concave). We shall prove that, for $i \in \Lambda$, $h \in [0, h_0]$, and almost surely with respect to~$\mathcal{F}_i$,
\[ \frac{\partial^-}{\partial h_i} \P( Z_i \in \cross_{\ell}(2R,R)  | \mathcal{F}_i )  \ge (2c_{\text{Mills}})^{-1} \P(Z_i  \in \cross_{\ell}(2R,R), \tilde{Z}_i \notin \cross_{\ell}(2R,R) | \mathcal{F}_i )  \]
where recall that $\tilde{Z}_i$ is an independent copy of $Z_i$. This is enough to conclude the proposition since, by averaging over $\mathcal{F}_i$ and using Fatou's lemma, we have
\[ \frac{\partial^-}{\partial h_i} \P( f^{\epsilon,h}_r  \in \cross_{\ell}(2R,R) )  \ge (4c_{\text{Mills}})^{-1}I_i , \]
and the result follows by the multivariate chain rule.

It remains to fix $i \in \Lambda$, condition on $\mathcal{F}_i$, and prove that $d_i \ge (2c_{\text{Mills}})^{-1} p_i$, where
\[  d_i = \frac{\partial^-}{\partial h_i} \P( Z_i \in \cross_{\ell}(2R,R)  | \mathcal{F}_i )  \]
and
\[  p_i = \P(Z_i  \in \cross_{\ell}(2R,R), \tilde{Z}_i \notin \cross_{\ell}(2R,R) | \mathcal{F}_i ) .\]

Since $g \ge 0$, the field $f^{\eps,h}_r$ is increasing with respect to $Z_i$, and so there exists a threshold $\omega \in [-\infty, \infty]$, measurable with respect to $\mathcal{F}_i$, such that $\{Z_i \in \cross_{\ell}(2R,R)\}$ occurs if $Z_i  > \omega$ and does not occur if $Z_i < \omega$ (whether or not $\{Z_i \in \cross_{\ell}(2R,R)\}$ occurs if $Z_i  = \omega$ may depend on $\mathcal{F}_i$, but it will not be relevant to the argument). We prove $d_i \ge (2c_{\text{Mills}})^{-1} p_i$ separately in three cases:

\smallskip
\noindent \textbf{Case 1: $\omega > h$.}
In this case we use that
\[  p_i \le  \mathbb{P}( Z_i \ge \omega) =   \mathbb{P}( \textrm{Ber}^{\eps^2}_i = 1) \P(Y_i  + h \ge \omega )  =  \epsilon^2 \P ( Y_i \ge \omega - h )  \]
since $\{Z_i \ge \omega\} = \{\textrm{Ber}^{\eps^2}_i = 1, Y_i  + h \ge \omega \}$. Moreover,
\[d_i  =  \mathbb{P}( \textrm{Ber}^{\eps^2}_i = 1) \frac{d}{d h}  \P(  Y_i \ge \omega - h)  = \eps^2 \mu(\omega-h)    \ge  c_\text{Mills}^{-1} \eps^2 \mathbb{P}( Y_i \ge \omega -h)   \]
where in the last inequality we used the definition of $c_{\text{Mills}}$ and the fact that $\omega -  h > 0$. Hence we have proven $d_i \ge c_{\text{Mills}}^{-1} p_i$.

\smallskip
\noindent \textbf{Case 2: $\omega < 0$.} 
In this case we use that
\[  p_i \le  \mathbb{P}( \tilde{Z}_i \le \omega) =   \mathbb{P}( \textrm{Ber}^{\eps^2}_i = 1) \P(Y_i  + h \le \omega )  =  \epsilon^2 \P ( Y_i \le \omega - h ) \]
since $\{Z_i \le \omega\} = \{\textrm{Ber}^{\eps^2}_i = 1, Y_i  + h \le \omega \}$. Moreover, as in the previous case
\[d_i  =  \mathbb{P}( \textrm{Ber}^{\eps^2}_i = 1) \frac{d}{d h}  \P(  Y_i \ge \omega - h)  = \eps^2 \mu(\omega-h)    \ge  c_\text{Mills}^{-1} \eps^2 \mathbb{P}( Y_i \le \omega -h)   \]
where in the last inequality we used the definition of $c_{\text{Mills}}$ and that $\omega -  h < 0$. Hence we have proven $d_i \ge c_{\text{Mills}}^{-1} p_i$ also in this case.

\smallskip
\noindent \textbf{Case 3: $\omega \in [0, h]$.} In this case we use that, as in the first case,
\[  p_i \le  \mathbb{P}( Z_i \ge \omega) =   \mathbb{P}( \textrm{Ber}^{\eps^2}_i = 1) \P(Y_i  + h \ge \omega )  =  \epsilon^2 \P ( Y_i \ge \omega - h )  , \]
and, as in the second case,
\[d_i  =  \mathbb{P}( \textrm{Ber}^{\eps^2}_i = 1) \frac{d}{d h}  \P(  Y_i \ge \omega - h)  = \eps^2 \mu(\omega-h)    \ge  c_\text{Mills}^{-1} \eps^2 \mathbb{P}( Y_i \le \omega -h) .  \]
Hence 
\[ d_i \ge c_{\text{Mills}}^{-1} p_i   \times \frac{ \P ( Y_i \le \omega - h ) }{ \P ( Y_i \ge \omega - h )} . \]
Then since $0 \le h - \omega  \le h \le h_0$, by symmetry and the definition of $h_0$ we have
\[  \frac{ \P ( Y_i \le \omega - h) }{ \P ( Y_i \ge \omega - h)}  = \frac{ \P ( Y_i \ge h - \omega  ) }{ \P ( Y_i \le h - \omega  )} \ge  \frac{ \P ( Y_i \ge  h_0 ) }{ \P( Y_i \le h_0 ) }   \ge \frac{1}{2}  . \]
Hence we have proven $d_i \ge (2c_{\text{Mills}})^{-1} p_i $ in this case.
\end{proof}

Before proving Proposition \ref{p:di} let us recall the OSSS inequality. Let $(X_i)_{i \in \Lambda}$ be an arbitrary finite set of independent random variables. For an event $A$ and $i \in \Lambda$, the resampling influence $I_i$ of $X_i$ on $A$ is the probability that resampling $X_i$ modifies $\id_A$ (this coincides with the definition we gave above). Let $\mathcal{A}$ be a \textit{random algorithm} that determines~$A$, i.e.\ a random adapted procedure that iteratively reveals a subset of $X_i$ and outputs $\id_A$. The \textit{revealment} $\delta_i$ of $X_i$ is the probability that $X_i$ is revealed by the algorithm. Then the OSSS inequality \cite{o05} states that
\[ \text{\textup{Var}}[\id_A] \leq \frac{1}{2} \sum_{i \in \Lambda} \delta_i  I_i  . \]

\begin{proof}[Proof of Proposition \ref{p:di}]
We will use the OSSS inequality to prove that, for  $\ell \in \mathbb{R}$, $r \ge 1$, $R > 2r$, $\epsilon > 0$ and $h \ge 0$,
\begin{align}
\label{e:osss}
\frac{ \sum_{i \in \Lambda} I_i }{ 1 -  \P (f^{\epsilon,h}_r \in \cross_{\ell}(2R,R))}  \ge  \frac{ 2 \P (f^{\epsilon,h}_r \in \cross_{\ell}(2R,R))  }{\inf_{\bar r \in (r, R/2)} \left( 2 \bar r / R + \P( f^{\epsilon}_r  \in \arm_{-\ell}(r, \bar r) ) \right)  }.
\end{align}
Given \eqref{e:osss} the proposition follows since  
\begin{align*}
  \frac{\partial^-}{\partial h} \Big( - \log \big( 1 -  \P( f^{\epsilon,h}_{r} \in \cross_{\ell}(2R,R) ) \big) \Big) & \ge  \frac{  \frac{\partial^-}{\partial h} \P( f^{\epsilon,h}_{r} \in \cross_{\ell}(2R,R) )  }{1-\P(f^{\epsilon,h}_{r} \in \cross_{\ell}(2R,R) )}\\
  &  \ge   \frac{1}{4c_{\text{Mills}}} \frac{ \sum_{i \in \Lambda} I_i }{1-\P( f^{\epsilon,h}_{r} \in \cross_{\ell}(2R,R) )}    
    \end{align*}
where the second inequality is by Proposition \ref{p:russo}.
    
To prove \eqref{e:osss} we apply the OSSS inequality to the random variables $(Z_i)_{i \in \Lambda}$ defined in \eqref{e:z} and the following algorithm (see \cite{mv18}, and also \cite{ab18,bks99,ss10}). First fix a random horizontal line $L = \{y = \mathcal{U}\}$ where $\mathcal{U}(R)$ is an independent $[0, R]$-uniform random variable. Then reveal $Z_i$ for all $i \in \Lambda$ within a distance $r$ from the line $L$; this determines the value of $f^{\eps, h}_r$ within a distance of $r/2$ from $L \cap ([0, 2R] \times [0,R])$. Finally, iteratively reveal $Z_i$ for all $i \in \Lambda$ within a distance $r$ from a component of $\{ f^{\epsilon,h}_r + \ell \le 0\}$ that intersects $L \cap ( [0, 2R] \times [0, R])$, and terminate with value $0$ if a connected component is found to intersect both the top and bottom sides of $[0, 2R] \times [0, R]$, and with value $1$ if no such component is found. 

\smallskip
The key properties of this algorithm are:
\begin{itemize}
\item The algorithm determines $\{f^{\epsilon,h}_r \in \cross_{\ell}(2R,R)\}$ almost surely; and
\item A coordinate $Z_i$ is revealed if and only if there exists a connected component of $\{ f^{\epsilon,h}_r + \ell \le 0\}$ that connects the ball of radius $r$ centred at $i$ to the line $L$.
\end{itemize}
 
Let us give an upper bound for the revealments $\delta_i$ of this algorithm. By symmetry, $\delta_i$ is bounded above by the probability that a connected component of $\{ f^{\epsilon,-h}_r - \ell \ge 0\}$ connects the ball of radius $r$ centred at $i$ to $L$. Since for any $\bar r < R/2$ the distance between $i$ and $L$ is larger than $\bar r$ with probability greater than $1-2 \bar r/R$, by averaging over $L$ we have
\[ \delta_i \le 2 \bar r / R + \P( f^{\epsilon,-h}_r  \in \arm_{-\ell}(r, \bar r) ) , \]
valid for any $\bar r \in (r, R/2)$. Then since  $\P( f^{\epsilon, -h}_r  \in \arm_{-\ell}(r, \bar r) ) )$ is decreasing in $h$, in fact 
\[ \delta_i \le 2 \bar r / R + \P( f^{\epsilon}_r  \in \arm_{-\ell}(r, \bar r) ) . \]
Applying the OSSS equality gives
\begin{align*}
 \textrm{Var} \big[ \id_{\{f^{\epsilon,h}_r \in \cross_{\ell}(2R,R)\}} \big] & =   \P \big(f^{\epsilon,h}_r \in \cross_{\ell}(2R,R) \big) \big(1 -  \P (f^{\epsilon,h}_r \in \cross_{\ell}(2R,R)) \big) \\
& \le  \frac{1}{2} \left( 2 \bar r / R + \P( f^{\epsilon}_r  \in \arm_{-\ell}(r, \bar r) ) \right) \sum_{i \in \Lambda} I_i    , 
\end{align*}
and rearranging gives \eqref{e:osss}.
\end{proof}

\subsection{The phase transition}
We are now ready to prove Theorem \ref{t:sharp}. Before giving the proof, let us describe its outline. Define the positive exponents 
\begin{equation}
\label{e:zeta}
\gamma \in (0, 1) \, , \quad \zeta = \frac{\gamma c_\text{Arm} + 1}{c_{\text{Arm}} + 1} \in (\gamma , 1) \quad \text{and} \quad \xi \in (0,1 - \zeta) 
\end{equation}
where $c_{\text{Arm}} > 0$ is the constant appearing in Theorem \ref{t:arm}. It is easy to check that such a choice is possible, and note that $\zeta$ has been defined to satisfy
\begin{equation}
\label{e:zeta2}
1 - \zeta =  (\zeta-\gamma)c_{\text{Arm}} .
\end{equation}
For the remainder of the section we also define the polynomial scales 
\[ r = r_R = R^\gamma \to \infty \ ,  \quad \bar r = \bar r_R  = R^\zeta \to \infty \quad \text{and} \quad h = h_R = R^{-\xi} \to 0 . \]
We prove Theorem \ref{t:sharp} in two steps. First we use the differential inequality in Proposition~\ref{p:di}, combined with a `sprinkling' procedure (i.e.\ a small raising of the level, see Lemma~\ref{l:sprink}), to prove Theorem \ref{t:sharp} for the field $f_r^{h}$. More precisely, we prove that for $\ell > 0$ there are $c_1, c_2 > 0$ such that for all $R \ge 1$,
\begin{equation}
\label{e:step1}
 \P \big( f_r^h \in \cross_{\ell}(2R,R) \big) \ge 1 -  c_1 e^{- R^{c_2}} . 
\end{equation}
Then we use a second sprinkling to eliminate the parameters $h$ and $r$ at the cost of a second small raising of the level.

\smallskip
We first state the `sprinkling' procedure:

\begin{lemma}[Sprinkling]
\label{l:sprink}
For every $\delta > 0$ there exist $c_1, c_2 > 0$ such that, for $R \ge 1$ and an increasing crossing event $E$ supported on $B(4R)$,
\[\P ( f +\delta  \in E ) \ge  \max\{  \P (f_r  \in E ) ,  \P (f_r^h  \in E)  \} - c_1 e^{-  R^{c_2} }  \] 
and
\[ \P( f_r + \delta \in E ) \ge \P(f \in E) - c_1 e^{-R^{c_2}} .  \]
\end{lemma}
\begin{proof}
By adjusting the constants $c_1, c_2 > 0$ it is sufficient to prove the statements for sufficiently large $R$.

First, by Lemma \ref{l:h} (setting $t = R^{\xi'}$ for some $\xi' \in (0, \xi)$), there are $c_1,c_2 > 0$ such that
\[ \P (    \| f_r - f_r^h \|_{B(4R), \infty} > R^{\xi'-\xi}) \le  c_1 R^{2} e^{- c_2 R^{\xi'}} . \]
Since $\xi'  < \xi$, this gives
\begin{equation}
\label{e:pert1}
\P ( \| f_r - f_r^h \|_{B(4R), \infty} > \delta ) \le   e^{-R^{c_3}}
\end{equation}
for some $c_3 > 0$ and sufficiently large $R$.

Second, by the second item of Lemma \ref{l:trunc} (setting $t = R^{\gamma'}$ for some $\gamma' \in (0, \gamma)$), there are $\alpha > 3$ and $c_4,c_5> 0$ such that
\[ \P ( \| f - f_r \|_{B(4R), \infty} \ge  c_4 R^{2\gamma'} r^{1-\alpha} (\log r)^2 ) < c_4 R^2 e^{- c_5 R^{\gamma'}}  \]
for sufficiently large $R$. Since $\gamma' < \gamma$ and we assume $\alpha > 3$, this gives
\begin{equation}
\label{e:pert2}
\P ( \| f - f_r \|_{B(4R), \infty} > \delta ) \le  e^{- c_6 R^{\gamma'}}  
\end{equation}
for some $c_6 > 0$ and sufficiently large $R$. 

The lemma then follows from the simple fact that, if $f_1$ and $f_2$ are fields on $\R^2$ such that $\P(\|f_1-f_2\|_\infty \ge \delta) < s$ for some $\delta,s>0$, and $A$ is an increasing event, then 
 \begin{equation*}
  \P(f_1 + \delta \in A) \ge \P( \{f_2 \in A\} \cap \{\|f_1-f_2\|_\infty \le \delta\}  ) \ge \P(f_2 \in A) -  s . \qedhere
  \end{equation*}
 \end{proof}

\begin{proof}[Proof of Theorem \ref{t:sharp}]
We begin by establishing \eqref{e:step1}. Applying Proposition \ref{p:di} (note that $r \ll \bar r \ll R$) there is a $c_1 > 0$ such that, for $h' \in [0, h]$, $\varepsilon > 0$, and sufficiently large $R$, 
\[ \frac{\partial^-}{\partial h'} \Big( - \log \big(1  - \P (  f^{\epsilon,h'}_r \in \cross_{\ell/2}(2R,R) ) \big) \Big)   \ge \frac{  c_1 \P (f^{\epsilon,h'}_r \in \cross_{\ell}(2R,R))}{  2 \bar r / R + \P( f^{\epsilon}_r \in  \arm_{-\ell}(r, \bar r) )   }  \]
Integrating both sides of this inequality over $h' \in [0, h]$ gives
\[ \frac{1 - \P \big(  f^{\epsilon,h}_r \in \cross_{\ell}(2R,R) \big)}{1  - \P (  f^{\epsilon,0}_r \in \cross_{\ell}(2R,R) )}  \le e^{- c_1  h ( 2R^{\zeta-1} + \P( f^{\epsilon}_r \in  \arm_{-\ell}(r, \bar r) )^{-1}  }    .  \]
Taking $\eps \to 0$, by Lemmas~\ref{l:dis} and \ref{l:prop} we deduce
\[ \frac{1 - \P \big(  f^h_r \in \cross_{\ell}(2R,R) \big)}{1  - \P (  f_r \in \cross_{\ell}(2R,R) )}  \le e^{- c_1 h ( 2R^{\zeta-1} + \P( f_r \in  \arm_{-\ell/2}(r, \bar r) )^{-1}  }    .  \]
By Lemma \ref{l:sprink}, Theorem \ref{t:arm} and \eqref{e:zeta2}, there are $c_2, c_3 > 0$ such that
\begin{align*}
\P (f_r \in \arm_{-\ell}(r, \bar r) ) & \le \P (f \in \arm_0(r, \bar r) ) + e^{- R^{c_2} }  \\
& \le  c_3 (r/\bar r)^{c_{\text{Arm}}} = c_3 R^{(\gamma-\zeta)c_{\text{Arm}}}   = c_3 R^{\zeta-1}   
\end{align*}
for sufficiently large $R$, and similarly by  Lemma \ref{l:sprink} and Theorem \ref{t:rsw}, 
\[ \P ( f_r \in \cross_{\ell}(2R,R) ) \ge \P (f \in \cross_0(2R,R) ) - e^{- R^{c_2}} \ge c_4 \]
for some $c_4 > 0$. Hence we deduce for some $c_5 > 0$,
\[  \P \big(  f^h_r \in \cross_{\ell}(2R,R) \big) \ge 1 - (1-c_4) e^{- c_5  h R^{1- \zeta}  }  =   1 - (1-c_4) e^{- c_5 R^{-\xi} R^{1-\zeta}  }    .  \]
Since $1-\zeta-\xi > 0$ by the choice of parameters in \eqref{e:zeta}, we conclude~\eqref{e:step1}.

To finish the proof we apply Lemma \ref{l:sprink} a second time. Together with \eqref{e:step1} this gives 
\begin{align*}
\P \big( f  \in \cross_{2\ell}(2R,R) \big)   \ge  \P \big( f^h_r \in \cross_{\ell}(2R,R) \big)  -  e^{-R^{c_6}}   \ge 1 - 2 e^{-R^{c_6}}  
\end{align*}
for some $c_6 > 0$ and sufficiently large $R$, which concludes the proof.
\end{proof}

\medskip

\section{The phase transition with respect to the intensity}
\label{s:c2}

In this section we show how the arguments in Section \ref{s:c1} can be adapted to prove the phase transition with respect to the intensity of positive marks without the log-concavity assumption \eqref{a:lc}. Throughout this section we assume Condition \ref{c:1} holds for a parameter $\alpha > 4$, and that there is a $\eta_0 > 0$ such that \eqref{a:bd} holds for $f^\eta$ for all $\eta \in [0, \eta_0]$. 

\smallskip
As in Section \ref{s:c1}, the main step is to establish the rapid convergence of crossing probabilities for $2R \times R$ rectangles, although in this case we prove only polynomial convergence:

\begin{theorem}
\label{t:sharp2}
For every $\eta > 0$ there exist $c_1,c_2 > 0$ such that, for every $R \ge 1$, 
\[  \P \left( f^\eta \in \cross_{0}(2R,R) \right) \ge 1 -  c_1 R^{-c_2}  . \]
\end{theorem}

\begin{proof}[Proof of Theorem \ref{t:main2}]
Since we assume $g \ge 0$, the field $f^\eta$ is increasing in $\eta$, and so the boundedness of the components of $\{f^\eta \ge 0\}$ for $\eta \le 0$ follows from the equivalent statement about $\{f \ge 0\}$ in Theorem \ref{t:main1}.

The fact that $\{f^\eta \ge 0\}$ has a unique unbounded component follows from Theorem \ref{t:sharp2} in the same way as in the proof of Theorem \ref{t:main1}; the fact that we have only polynomial, instead of stretched-exponential, convergence is irrelevant (in fact all that we need is that $1-\P \left( f^\eta \in \cross_{0}(2R,R) \right)$ is summable over the sequence $R_k = 2^k$).
\end{proof}

It remains to prove Theorem \ref{t:sharp2}. As in Section \ref{s:c1} the crux is to prove a Russo-type inequality analogous to Proposition \ref{p:russo}. Let us first introduce a truncated discretised approximation $f^{\eps, \eta}_r$ of $f^\eta$; this will play the same role as $f^{\eps, h}_r$ in Section \ref{s:c1}. For $\eps > 0$, $(\eta_i)_{i \in \eps \Z^2}$ a real-valued vector, and $r > 0$ define
\[      f^{\eps,\eta}_r(x) = \sum_{i \in \eps \Z^2} \textrm{Ber}^{\eps^2}_i  \Big( |Y_i| \id_{\{ \mathcal{U}_i \le 1/2 + \eta_i  \}}  -  |Y_i|  \id_{\{ \mathcal{U}_i \ge  1/2 + \eta_i \}} \Big) g_r(x-i) ,  \]
where $\textrm{Ber}^{\eps^2}_i$ and $g_r$ are as in the definition of $f^{\eps,h}_r$. In the case $\eps = 0$ we define instead
\[     f^\eps_r(x) =  f^{0,\eta}_r(x) = \sum_{i \in \mathcal{P}}  \Big( |Y_i| \id_{\{ \mathcal{U}_i \le 1/2 + \eta_i  \}}  -  |Y_i|  \id_{\{ \mathcal{U}_i \ge  1/2 + \eta_i \}} \Big) g_r(x-i) .  \]
If $\eta \in \R$ we understand this as setting $\eta_i \equiv \eta$.

\smallskip  For $R > 0$, the event $\{ f^{\epsilon,\eta}_r \in \cross_0(2R,R) \}$ is measurable with respect to the finite set of independent random variables
\begin{equation}
\label{e:z2}
(Z_i )_{i \in \Lambda}= (\textrm{Ber}^{\eps^2}_i ( |Y_i| \id_{\{ \mathcal{U}_i \le 1/2 + \eta_i  \}}  -  |Y_i|  \id_{\{ \mathcal{U}_i \ge  1/2 + \eta_i \}} )  )_{i \in \Lambda} , 
\end{equation}
where $\Lambda \subset \eps \Z^2$ is defined as in \eqref{e:z}. For each $i \in \Lambda$ define the resampling influence
\[     I_i = \frac12 \E[ \P( Z_i \in \cross_0(2R,R) , \tilde{Z}_i \notin \cross_0(2R,R) | \mathcal{F}_i  ) ]   , \]
 where $\tilde{Z}_i$ is an independent copy of $Z_i$, and $\mathcal{F}_i$ is the $\sigma$-algebra generated by $(Z_j)_{j \in \Lambda \setminus \{i\}}$.

\begin{proposition}[Russo-type inequality; second version]
\label{p:russo2}
For all $r, R, \varepsilon > 0$ and $\eta \in [-1/2,1/2]$,
\[  \frac{\partial^-}{\partial \eta}  \P \big( f^{\epsilon,\eta}_r \in \cross_0(2R,R) \big)  \ge \frac{1}{2}  \sum_{i \in \Lambda} I_i  . \]
\end{proposition}

\begin{proof}
As in the proof of Proposition \ref{p:russo}, by Fatou's lemma and the multivariate chain rule it suffices to prove that, for every $i \in \Lambda$, $\eta \in [-1/2,1/2]$, and almost surely with respect to $\mathcal{F}_i$, we have $d_i \ge p_i$ where
\[ d_i =  \frac{\partial^-}{\partial \eta_i}  \P \big(Z_i \in \cross_0(2R,R) | \mathcal{F}_i \big)\]
and
\[ p_i= \P( Z_i \in \cross_0(2R,R) , \tilde{Z}_i \notin \cross_0(2R,R) | \mathcal{F}_i  ).\] 
Since $g \ge 0$ the field $f^{\eps,\eta}_r$ is increasing with respect to $Z_i$, and so there exists a threshold $\omega \in [-\infty, \infty]$, measurable with respect to $\mathcal{F}_i$, such that $\{Z_i \in \cross_0(2R, R)\}$ occurs if $Z_i > \omega$, does not occur if $Z_i < \omega$, and whether  $\{Z_i \in \cross_0(2R, R)\}$ occurs if $Z_i = \omega$ may depend on $\mathcal{F}_i$.  We prove $d_i \ge p_i$ separately in four cases:

\smallskip
\noindent \textbf{Case 1: $\omega > 0$ and $\{Z_i \in \cross_0(2R, R)\}$ occurs if $Z_i = \omega$.}
In this case we use that
\[ p_i \le \P(Z_i \ge \omega) = \P(\textrm{Ber}_i^{\eps^2} = 1) \P( |Y_i| \ge \omega) = \eps^2 \P(|Y_i| \ge \omega)  \]
since $\{Z_i \ge \omega\} = \{ \textrm{Ber}_i^{\eps^2} = 1,  \mathcal{U}_i \le 1/2 + \eta_i , |Y_i| \ge \omega \}$. Moreover,
\[ d_i = \P(\textrm{Ber}_i^{\eps^2} = 1) \P( |Y_i| \ge \omega)  = \eps^2 \P(|Y_i| \ge \omega). \]

\smallskip
\noindent \textbf{Case 2: $\omega \ge 0$ and $\{Z_i \in \cross_0(2R, R)\}$ does not occur if $Z_i = \omega$.} In this case we use that
\[ p_i \le \P(Z_i > \omega) = \P(\textrm{Ber}_i^{\eps^2} = 1) \P( |Y_i| > \omega) = \eps^2 \P(|Y_i| > \omega)  \]
since $\{Z_i > \omega\} = \{ \textrm{Ber}_i^{\eps^2} = 1,  \mathcal{U}_i \le 1/2 + \eta_i , |Y_i| > \omega \}$. Moreover, in this case
\[ d_i = \P(\textrm{Ber}_i^{\eps^2} = 1) \P( |Y_i| > \omega)  = \eps^2 \P(|Y_i| > \omega). \]

\smallskip
\noindent \textbf{Case 3: $\omega \le 0$ and $\{Z-i \in \cross_0(2R, R)\}$ occurs if $Z_i = \omega$.}  In this case we use that 
\[ p_i \le \P(\tilde{Z}_i < \omega) = \P(\textrm{Ber}_i^{\eps^2} = 1) \P( -|Y_i| < \omega) = \eps^2 \P(-|Y_i| < \omega)  \]
since $\{Z_i < \omega\} = \{ \textrm{Ber}_i^{\eps^2} = 1,  \mathcal{U}_i \ge 1/2 + \eta_i , - |Y_i| < \omega \}$.  Moreover, in this case,
\[ d_i = \P(\textrm{Ber}_i^{\eps^2} = 1) \P( -|Y_i| < \omega)  = \eps^2 \P(-|Y_i| < \omega) .\]

\smallskip
\noindent \textbf{Case 4: $\omega < 0$ and $\{Z_i \in \cross_0(2R, R)\}$ does not occur if $Z_i = \omega$.} Similarly we have
\[ p_i \ge \P(\tilde{Z}_i \le \omega) = \P(\textrm{Ber}_i^{\eps^2} = 1) \P( \mathcal{U}_i \ge  \P(- |Y_i| \le \omega) = \eps^2 (1/2 - \eta) \P(-|Y_i| \ge \omega) \]
and
\[ d_i = \P(\textrm{Ber}_i^{\eps^2} = 1) \P( -|Y_i| \le \omega)  = \eps^2 \P(- |Y_i| \le \omega) .\]

\noindent In all cases we have $d_i \ge p_i$, which completes the proof.
\end{proof}

Combining the Russo-type inequality in Proposition \ref{p:russo2} with the OSSS inequality we deduce the following differential inequality:

\begin{proposition}[Differential inequality]
\label{p:di2}
There exists $c_1 > 0$ such that, for all $r \ge 1$, $R > 2r$, $\epsilon > 0$ and $\eta \in [-1/2,1/2]$,
\begin{align*}
& \frac{\partial^-}{\partial \eta}  \Big( - \log \big( 1 - \P( f^{\epsilon,\eta}_r \in \cross_0(2R,R) ) \big)  \Big)  \\
& \qquad \ge \frac{ \P (f^{\epsilon,\eta}_r \in \cross_0(2R,R))}{ \inf_{\bar r \in (r, R/2)} \left( 2 \bar r / R + \P( f^\epsilon_r \in  \arm_{-\ell}(2r, \bar r) ) \right)  } .
\end{align*}
\end{proposition}
\begin{proof}
This is proven exactly as Proposition \ref{p:di}, replacing Proposition \ref{p:russo} with Proposition~\ref{p:russo2}. Note that it is important that $f^{\eps,\eta}_r$ is increasing in $\eta$.
\end{proof}

We can now finish off the proof of Theorem \ref{t:sharp2}. One difference with the proof of Theorem~\ref{t:sharp} is that we cannot rely on the sprinkling procedure in Lemma \ref{l:sprink} to compare $f^\eta_r$ to $f$ since we do not want to vary the level. Instead we use the weaker Lemma \ref{l:sprink2} below.

\smallskip
Recall that Condition \ref{c:1} holds for a parameter $\alpha > 4$, and define the positive exponents
\begin{equation}
\label{e:zeta3}
\gamma \in \Big( 0,  \frac{2}{\alpha-2} \Big)  \quad \text{and} \quad \zeta = \frac{\gamma c_\text{Arm} + 1}{c_{\text{Arm}} + 1} \in (\gamma , 1) 
\end{equation}
where $c_{\text{Arm}} > 0$ is the constant appearing in Theorem \ref{t:arm}. As in \eqref{e:zeta}, $\zeta$ has been defined to satisfy $1 - \zeta =  (\zeta-\gamma)c_{\text{Arm}}$. For the remainder of the section we use the polynomial scales 
\[ r = r_R = R^\gamma \to \infty   \qquad  \text{and} \qquad \bar r = \bar r_R  = R^\zeta \to \infty  . \]
Recall also that \eqref{a:bd} holds for the field $f^\eta$ for all $\eta \in [0, \eta_0]$.

\begin{lemma}
\label{l:sprink2}
For every $\eta \in [0, \eta_0]$ there exist $c_1, c_2 > 0$ such that, for $R \ge 2$ and an increasing crossing event $E$ supported on $B(4R)$,
\[| \P ( f^\eta \in E ) - \P( f^\eta_r \in E )| < c_1 R^2  r^{2-\alpha} (\log R)^8 . \]
\end{lemma}
\begin{proof}
This is proven analogously to Proposition \ref{p:perturbation} for the field $f^\eta$ in place of~$f$, substituting the first item of Lemma \ref{l:trunc} in place of the second item in the proof; this is necessary since $f^\eta$ is not a symmetric shot noise field. Note that to apply Lemma~\ref{l:trunc} we rely on the assumption that $f^\eta$ satisfies \eqref{a:bd}.
\end{proof}
 
\begin{proof}[Proof of Theorem \ref{t:sharp2}]
Since $f^\eta$ is increasing in $\eta$, it is sufficient to prove the result for $\eta \in (0, \eta_0)$. Applying Proposition \ref{p:di2} (note that $r \ll \bar r \ll R$) there is a $c_1 > 0$ such that, for $\eta' \in [0, \eta]$, $\eps > 0$ and sufficiently large $R$,
\[ \frac{\partial^-}{\partial \eta'} \Big( - \log \big(1  - \P (  f^{\epsilon,\eta'}_r \in \cross_{0}(2R,R) ) \big) \Big)   \ge \frac{  c_1 \P (f^{\epsilon,\eta'}_r \in \cross_0(2R,R))}{  2 \bar r / R + \P( f^{\epsilon}_r \in  \arm_0(r, \bar r) )   } .  \]
Integrating both sides of this inequality over $\eta' \in [0, \eta]$ gives
\[ \frac{1 - \P \big(  f^{\epsilon,\eta}_r \in \cross_{0}(2R,R) \big)}{1  - \P (  f^{\epsilon,0}_r \in \cross_0(2R,R) )}  \le e^{- c_1  \eta ( 2R^{\zeta-1} + \P( f^{\epsilon}_r \in  \arm_0(r, \bar r) )^{-1}  }    .  \]
Taking $\eps \to 0$, by Lemmas~\ref{l:dis} and \ref{l:prop} we deduce
\[ \frac{1 - \P \big(  f^\eta_r \in \cross_0(2R,R) \big)}{1  - \P (  f_r \in \cross_0(2R,R) )}  \le e^{- c_1 \eta ( 2R^{\zeta-1} + \P( f_r \in  \arm_0(r, \bar r) )^{-1}  }    .  \]
By Lemma \ref{l:sprink2} (recall that $f = f^0$ and $f_r = f^0_r$) and Theorem \ref{t:arm}, there are $c_2, c_3, c_4, c_5> 0$ such that
\begin{align*}
\P (f_r \in \arm_0(r, \bar r) ) & \le \P (f \in \arm_0(r, \bar r) ) +  c_2 R^2  r^{2-\alpha} (\log R)^8    \\
& \le  c_3 (r/\bar r)^{c_{\text{Arm}}} +  c_2 R^2  r^{2-\alpha} (\log R)^8   \\ 
&   = c_3 R^{\zeta-1}   +  c_2 R^{2 + \gamma(2-\alpha)} (\log R)^8 \le c_4 R^{-c_5}   
\end{align*}
for sufficiently large $R$, where we used that $\gamma < 2/(\alpha-2)$ and $1 - \zeta =  (\zeta-\gamma)c_{\text{Arm}}$. Similarly by  Lemma~\ref{l:sprink2} and Theorem \ref{t:rsw}, 
\[ \P ( f_r \in \cross_0(2R,R) ) \ge \P (f \in \cross_0(2R,R) ) -  c_2 R^{2 + \gamma(1-\alpha)} (\log R)^8 \ge c_6 \]
for some $c_6 > 0$. Hence we deduce for some $c_7 > 0$,
\[  \P \big(  f^\eta_r \in \cross_0(2R,R) \big) \ge 1 - (1-c_6) e^{- c_7  \eta R^{-c_5}  }     \]
for sufficiently large $R$. Applying Lemma \ref{l:sprink2} a second time gives 
\begin{align*}
\P \big( f^\eta  \in \cross_0(2R,R) \big)   \ge  \P \big( f^\eta_r \in \cross_0(2R,R) \big)  -  c_8 R^2  r^{2-\alpha} (\log R)^8  \ge 1 - c_9 R^{-c_{10}}  
\end{align*}
for some $c_8, c_9, c_{10} > 0$ and sufficiently large $R$, which concludes the proof.
\end{proof}

\medskip

\appendix 
\section{}
\label{s:a}

\subsection{Concentration inequalities} 
\label{a:con}
In this section we prove concentration inequalities for shot noise fields with unbounded marks, based on the work of Reynaud-Bouret \cite{Rb} on Poisson stochastic integrals.
\smallskip
Our result concerns shot noise fields in arbitrary dimension, i.e.\ we consider the field
\[   f(x) = \sum_{i \in \mathcal{P}} Y_i g(x-i)  \]
where $g(x) \in L^1(\mathbb{R}^d)$, $\mathcal{P}$ is a Poisson point process on $\mathbb{R}^d$ with unit intensity, and $(Y_i)_{i \in \mathcal{P}}$ are i.i.d.\ random variables drawn from a mark distribution $\mu$ with finite mean.

\smallskip
We assume the following conditions:

\begin{condition}
\label{c:3} $\,$
\begin{itemize}
\item (Kernel; depends on a parameter $\alpha > d$) The kernel has polynomial decay, i.e.\ there exists a parameter $\alpha >d$ and a constant $c_1 > 0$ such that, for all $x \in \R^d$,
\[ | g(x)| < c_1(1+|x|)^{-\alpha} .\]
\item (Mark; depends on a parameter $\beta \in \R^+ \cup \{\infty\} $) The mark distribution has a stretched-exponential tail, i.e.\ either there are $\beta>0$ and $c_2>0$ such that, for $u \ge 0$,
\[  \mu ( (-\infty ,-u) \cup (u,\infty )) \le  c_2 \exp(-u^{\beta}) , \]
or else $\mu$ is supported on $[-1,1]$ (we refer to this case as `$\beta =\infty$'). 
\end{itemize}
\end{condition}
To state our result we define, for $x \in \mathbb{R}^d$,
\[ \varphi(x) = (1+ \max\{0,\log |x| \} )^{\frac{2}{\beta} } , \]
with the convention that, if $\beta = \infty$, $t^{1/\beta} = 1$ for all $t > 0$ (so in particular $\varphi(x) = 1$ in this case). Moreover, for a function $h: \mathbb{R}^d \to \mathbb{R}$, we introduce the auxiliary functions
\[  \tilde h(x)= \varphi(x) h(x) \quad \text{and} \quad \hat h(x) =\sup_{y\in Q} |\tilde{h}(x+y)| , \]
where $Q=[-1/2,1/2]^d$. The functions $\tilde g$ and $\hat g$ are majorants of $g$ which, assuming $g$ is sufficiently regular, asymptotically differ from $|g|$ by a logarithmic term. We also define
\[ \kappa = \kappa(\mu) =  \begin{cases} 2 + c_2  \int_{\mathbb{R}^{d}} \exp(-\varphi(x)^\beta /2) \,dx , & \beta < \infty, \\  2, & \beta = \infty,  \end{cases}  \]
where $c_2 > 0$ is the constant in Condition \ref{c:3}.

\begin{proposition}[Concentration with unbounded marks] 
\label{p:ldb2}
Suppose Condition \ref{c:3} holds with parameters $\alpha > d$ and $\beta \in \R^+ \cup \{\infty\}$, and recall the convention $t^{1/\beta} = 1$ if $t > 0$ and $\beta = \infty$. Then:
\begin{enumerate}
\item  For all $s,t \ge 1$,
\begin{align}
\label{eq:concentr-sup}
\mathbb{P} \bigg(\|f\|_{sQ, \infty} \ge t^{\frac{1}{\beta }} \Big(\|\hat{g}\|_{{1}}+  \sqrt{2t}  \|\hat {g}\|_{{2}}+\frac{t}{3}\|  \hat g\|_{\infty } \Big) \bigg) \le  (s+1)^{2} \kappa e^{-t /2}.
\end{align}  
\item Assume that $\mu $ is symmetric. Then, for all $t \ge  1$,
\begin{align}
\label{eq:concentr}
\mathbb{P} \bigg(|f(0)| \ge t^{\frac{1}{\beta }} \Big( \sqrt{2t} \|\tilde g\|_{{2}}+\frac{t}{3}\|\tilde g\|_{\infty } \Big) \bigg) \le \kappa e^{-t /2}.
\end{align}
Suppose furthermore that $g$ is in $C^1(\R^d)$, and there exists $c_3 > 0$ such that, for all $i=1,\ldots ,d$ and $x \in \R^d$,
\[\partial_{i}g(x) \le c_3(1+|x|)^{-\alpha-1} . \]
Then for all $s,t \ge 1$,
\begin{align}
\label{eq:concentr-grad}
& \mathbb{P} \bigg(\|f \|_{sQ, \infty} \ge dt^{\frac{1}{\beta }} \Big(\sqrt{2t}\|\tilde g\|_{{2}}+\frac{t}{3}\|\tilde g\|_{\infty }+ \|\widehat{\nabla g}\|_{{1}}+\sqrt{2t}\|\widehat{\nabla g}\|_{{2}}+\frac{t}{3}\|\widehat{\nabla g}\|_{\infty } \Big) \bigg) \\
\nonumber &  \qquad \qquad \le (s+1)^{2}(d+1)\kappa e^{-t /2} .
\end{align}
\end{enumerate}
\end{proposition}
 
To illustrate these bounds, we present the following application to $r$-dependent approximations of $f$ as defined in Section \ref{ss:pert}. For $r>1$, let $f^r$ be defined as in the proof of Lemma \ref{l:trunc}, i.e.\ $f^r$ is the shot noise field with mark distribution $\mu$ and kernel 
\[ g^r(x) = g(x) ( 1- \chi(|x|/r) )  \]
where $\chi$ is an infinitely differentiable cutoff function as in \eqref{e:gr}.

\begin{proposition}
\label{ex:concentr}
Suppose Condition \ref{c:3} holds with parameters $\alpha > d$ and $\beta \in \R^+ \cup \{\infty\}$, and assume for simplicity that $g(x)= c(1+|x|)^{-\alpha }$ for some $\alpha >d$ and $c > 0$. Then:
\begin{enumerate}
\item There are $c_1,c_2, c_3>0$ such that, for $r,s \ge 2$,
\begin{align*}
\mathbb{P} \big(\|f^r\|_{sQ, \infty} \ge & c_1r^{d-\alpha}(\log r)^{c_2} \big) 
\le   c_3 s^2 \exp(-(\log r)^{2}/2) .
\end{align*}
\item If in addition $\mu$ is symmetric, there are $c_1,c_2,c_3>0$ such that, for $r,s \ge 2$,
\begin{align*}
\mathbb{P} \big(\|f^r\|_{sQ, \infty} \ge & c_1 (r^{d-1-\alpha}+r^{\frac{d}{2}-\alpha }) (\log r)^{c_2} \big) 
\le   c_3 s^2 \exp(-(\log r)^{2}/2) .
\end{align*}
\end{enumerate}
In particular, these probabilities decay faster than any polynomial as $r \to \infty$. 
\end{proposition}

\begin{proof}
Observe that, as $r \to \infty$, for some constants $c'_{i} > 0$,
\[ \|g^r\|_{\infty }\sim c'_{1} r^{-\alpha } \, , \quad \|g^r\|_{{1}}\sim c_{2}' r^{d-\alpha } \, , \quad \|g^r\|_{{2}}\sim c_{3}' r^{d/2-\alpha }, \]
\[ \| {\nabla g^r}\|_{\infty }\sim c_{4}' r^{-\alpha -1} \, , \quad \| {\nabla g^r}\|_{{1}}\sim c_{5}' r^{d-\alpha -1} \quad \text{and} \quad \| {\nabla g^r}\|_{{2}}\sim c_{6}' r^{d/2-\alpha -1} , \]
and the corresponding norms of $\tilde{g^r}, \hat{g^r}$ and $\widehat{\nabla g^r}$ also decay with the same respective powers (up to logarithmic factors).  In particular, the dominating terms among the norms in \eqref{eq:concentr-sup} and \eqref{eq:concentr-grad} are respectively of order $r^{d - \alpha}$ and $r^{\max\{ d-1-\alpha, d/2-\alpha \}}$ (up to logarithmic factors). Then set $t=(\log r)^2$ and apply \eqref{eq:concentr-sup} and \eqref{eq:concentr-grad} respectively.
\end{proof}

\newcommand{\x}{{\bf x}}

\begin{proof}[Proof of Proposition \ref{p:ldb2}]
Let $t \ge 1$. We begin by defining a suitable event that allows us to truncate the mark distribution. Denote by $(x_{i})_{i \ge 1}$ some enumeration of the points of $\mathcal{P}$, and remark that $(x_{i},Y_{i})_{i \ge 1}$ has the law of a Poisson point process $\mathcal{P}'$ with intensity measure  $dx\mu (dm)$ on $\mathbb{R}^{d}\times \mathbb{R}$, called the {\it marked point process}. Define $\tilde Y_{i}=Y_{i}/([x_{i}]t^{1/\beta })$ and $Z=\{(x,m):\;|m| \le \varphi(x)t^{1/\beta }\}$, and define the truncation event
\[ \Omega =  \{ |\tilde Y_{i}| \le 1 \text{ for all } i \ge 1\} = \{ \mathcal{P}' \cap Z^{c}=\emptyset \}  . \]
Remark that $\mathbb{P}(\Omega )=1$ if the marks are supported on $[-1,1]$. Otherwise, the Campbell-Mecke formula bounds the probability of the complement as
\begin{align}
\label{e:omega}
\notag  \mathbb{P}(\Omega^{c})&\le \mathbb{E}\Big(\sum_{(x_{i},m_{i})\in \mathcal  P'}\id_{\{ | m_{i} | >[x_{i}]t^{1/\beta }\}}\Big)\\
\notag   & \le  \sum_{i}\mathbb{P}(|\tilde Y_{i}|>1) = \int_{\mathbb{R}^{d}}\mathbb{P}(Y_{1}>\varphi(x)t^{1/\beta })dx\le c_2 \int_{\mathbb{R}^{d}}\exp(-\varphi(x)^{\beta }t )dx \\
 \qquad& \le c_2 \int_{\mathbb{R}^{d}}\exp \Big(-\frac{t}{2}(1+\varphi(x)^{\beta }) \Big)dx \le  c_2 e^{-t/2} \int_{\mathbb{R}^{d}}\exp(-\varphi(x)^{\beta }/2)dx .
 \end{align} 

We next recall the results of Reynaud-Bouret \cite{Rb} on Poisson stochastic integrals in the present context. Let $\mathcal{P}''$ be a Poisson point process with intensity $dx\mu (dm) \id_{\{(x,m)\in Z\}}$ on $\mathbb{R}^{d}\times \mathbb{R}$. For $h:\mathbb{R}^{d} \to \mathbb{R}$ bounded and measurable, introduce the auxiliary function 
 \[ \bar h(x,m)=mh(x)(\varphi(x)t^{1/\beta})^{-1}\id_{\{(x,m)\in Z\}}  \, , \quad (x,m) \in \mathbb{R}^d \times \mathbb{R} , \]
 and define
\begin{align*}
I_{h,\mu } =I_{h}=\sum_{(x,m)\in \mathcal{P}''}\bar h(x,m)
\end{align*}
where $\mu $ is the law of the marks. Still denoting by $\|\bar h\|_{p}$  the $L^{p}$ norm in $\mathbb{R}^{d}\times \mathbb{R}$ for $p\in [1,\infty ]$, \cite[Proposition 7]{Rb} states that
\begin{align}
\label{e:con1}
\mathbb{P} \Big(I_{h}\ge \mathbb{E}(I_{h})+\sqrt{2t}\|\bar h\|_{2} +\frac{t}{3}\|\bar h\|_{\infty } \Big) \le e^{-t}.
\end{align}
Since $\|\bar h\|_{2} \le \|h\|_{2}$ and $\|\bar h\|_{\infty } \le \|h\|_{\infty }$, it yields
\begin{align}
\label{e:con2}
\mathbb{P} \Big(I_{h} \ge \mathbb{E}(I_{h}) + \sqrt{2t} \|h\|_{2} +\frac{t}{3}\|h\|_{\infty} \Big)  \le  e^{-t}.
\end{align}

We are now ready to prove the claims \eqref{eq:concentr-sup}--\eqref{eq:concentr-grad}, beginning with \eqref{eq:concentr}.  First note that, for all $u \ge 0$, 
\begin{align*}
\mathbb{P} \big(f(0) \ge t^{1/\beta }u \big) & = \mathbb{P}\bigg(\sum_{(x_{i},Y_{i})\in \mathcal{P}'}Y_{i}t^{-1/\beta }g(x_{i}) \ge u \bigg) = \mathbb{P} \bigg(\sum_{(x_{i},Y_{i})\in \mathcal{P}'}\tilde Y_{i}\tilde g(x_{i}) \ge u \bigg) \\
& \le \mathbb{P} \bigg( \Big\{ \sum_{(x_{i},Y_{i})\in \mathcal{P}'}\tilde Y_{i}\tilde g(x_{i}) \ge u \Big\} \cap \Omega \bigg)+\mathbb{P}(\Omega ^{c}) \\
& \le \mathbb{P} \bigg(\sum_{(x_{i},Y_{i})\in \mathcal{P}''}\tilde Y_{i}\tilde g(x_{i}) \ge  u \bigg) +\mathbb{P}(\Omega ^{c}) \\
& = \mathbb{P}(I_{\tilde{g}} \ge u) + \mathbb{P}(\Omega ^{c}),
\end{align*} 
where the last inequality is proven by discretising $\mathcal  P$. Note also that $\mathbb{E}(I_{\tilde g})=0$ if the mark distribution $\mu$ is symmetric. Abbreviating $u_1 = \sqrt{2t} \|\tilde g\|_{2}+\frac{t}{3}\|\tilde g\|_{\infty }$ and applying \eqref{e:con2} gives that $\mathbb{P}(I_{\tilde{g}} \ge u_1) \le e^{-t}$. Hence, doing the same with $h=-\tilde g$,
\begin{align*}
\mathbb{P}(|f(0)|\ge t^{\frac{1}{\beta }}u_1)\le \mathbb{P}(f(0)\ge t^{\frac{1}{\beta }}u_1)+\mathbb{P}(-f(0)\ge t^{\frac{1}{\beta }}u_1)\le 2e^{-t}+\mathbb{P}(\Omega ^{c}).
\end{align*} 
Combining with \eqref{e:omega} yields \eqref{eq:concentr}.

We turn to \eqref{eq:concentr-sup}. Similarly to above, for all $u \ge 0$,
\begin{align*}
\mathbb{P} \Big(\sup_{x \in Q}| f(x)|\ge t^{1/\beta } u \Big) & =\mathbb{P}\Big(\sup_{x \in Q} \Big| \sum_{(x_{i},Y_{i})\in \mathcal{P}'} \tilde Y_{i}\tilde g(x + x_{i}) \Big|\ge u \Big) \\
& \le \mathbb{P} \Big(\sum_{(x_{i},Y_{i})\in \mathcal{P}'} |\tilde Y_{i}|\sup_{x \in Q}|\tilde g(x+x_{i})|\ge u \Big)\\
& = \mathbb{P} \Big(\sum_{(x_{i},Y_{i})\in \mathcal{P}'} |\tilde Y_{i}| \hat{g}(x )\ge u \Big) \\
& \le  \mathbb{P}(I_{\hat{g},\mu '} \ge u) +\mathbb{P}(\Omega^{c}) 
\end{align*}
where $\mu '$ is the law of the $ |Y_{i}| $ (which has stretched exponential decay by Condition~\ref{c:3} but is not symmetric). Abbreviating $u_2=\|\hat{g}\|_{1}+\sqrt{2t}\|\hat{g}\|_{2}+\frac{t}{3}\|\hat g\|_{\infty }$ and applying \eqref{e:con2} gives that
\begin{align*}
\mathbb{P}(I_{\hat{g}} \ge u_2) \le e^{-t}.
\end{align*}
Combining with \eqref{e:omega} yields \eqref{eq:concentr-sup} for $s=1$. For larger $s$, cut $\lceil s\rceil Q$ into $\lceil s \rceil ^{2}$ cubes $Q_{k}$ of sidelength $1$. We have, for $u\ge 0$, 
by translational invariance
\begin{align*}
\mathbb{P}(\|f\|_{sQ,\infty}\ge u)\le \sum_{k=1}^{\lceil s\rceil^{2}}\mathbb{P}(\|f\|_{Q_{k},\infty}\ge u)\le \lceil s\rceil^{2}\mathbb{P}(\|f\|_{Q,\infty}\ge u).
\end{align*}  This yields \eqref{eq:concentr-sup} for general $s \ge 1$.

Finally, let us prove \eqref{eq:concentr-grad}. For $x\in sQ, |f(x)|\le |f(0)|+\frac{\sqrt{d}}{2}\sup_{x\in sQ} |\nabla f(x)|$. Hence we have, for $u\ge 0$,
\begin{align*}
\mathbb{P}(\|f \|_{sQ,\infty}\ge u) & \le  \mathbb{P} \Big(|f(0) |\ge \frac{u}{2}\text{\rm{ or }}\exists 1\le i\le d: \sup_{sQ}|\partial _{i}f|\ge \frac{u}{d} \Big)\\
& \le \mathbb{P} \Big(|f(0) |>\frac{u}{2} \Big)+\sum_{i=1}^{d}\mathbb{P} \Big(\|\partial _{i}f\|_{sQ,\infty}\ge \frac{u}{d} \Big) .
\end{align*}
Setting
\[ u=dt^{\frac{1}{\beta }} \Big(\sqrt{2t}\|\tilde g\|_{{2}}+\frac{t}{3}\|\tilde g\|_{\infty }+\| \widehat{ \nabla g}\|_{{1}}+\sqrt{2t}\|\widehat{\nabla g}\|_{{2}}+\frac{t}{3}\|\widehat{\nabla g} \|_{\infty } \Big) \]
and applying \eqref{eq:concentr} to the field $f$ and \eqref{eq:concentr-sup} to the field $\partial_i f$ (which is almost surely a shot noise field with kernel $\partial_i g$) we get the result.
\end{proof}

\subsection{Bounded density of shot noise fields }
\label{sec:density-power}

In this section we give examples of shot noise fields \eqref{e:ssn} which satisfy the bounded density assumption \eqref{a:bd} (see also Remark \ref{rk:kernel-motiv}).

\begin{proposition}
Consider the shot noise field \eqref{e:ssn} with arbitrary non-zero mark distribution $\mu$ and kernel that is either
\begin{equation}
\label{e:power}
 g(x)=(1+|x|)^{-\alpha}  \quad \text{or} \quad g(x)=(1+ | x | ^{2})^{-\alpha /2}
 \end{equation}
for some $\alpha >2$, or 
\begin{align}
\label{e:exp}
g(x)=\exp(- |x|^\beta )\quad \text{or} \quad g(x)=\exp(-(1+|x|^2)^{\beta /2})
\end{align} 
for some $\beta \in (0,1)$. Then the density of $(f(0),\nabla f(0))$ is continuous and bounded.
\end{proposition}
  
\begin{proof}
By the Riemann-Lebesgue lemma, the density of $(f(0),\nabla f(0))$ being continuous and bounded is implied by the integrability of its characteristic function 
\[ \varphi(u,v)=\exp \Big(\int_{}\left(\exp (im[ug(x)+ \langle v, \nabla g(x) \rangle ] )-1\right)dx\mu (dm) \Big) ,u\in \mathbb{R},v\in \mathbb{R}^{2} . \]
In fact if the mark distribution $\mu$ is symmetric, and so $\varphi$ takes values in $[0,1]$, these are equivalent. So it suffices to prove that $|\varphi(u,v)|$ is integrable on $\mathbb{R}\times \mathbb{R}^{2}$. 

Let $u>0,v\neq 0$. Introduce a parameter $\rho = \rho(u,v) >0$ whose value will be fixed later, and define $u_{\theta }=(\cos(\theta ),\sin(\theta )),\theta \in [0,2\pi ]$. Without loss of generality there is a $m_0>0$ such that $\mu ([m_{0},m_{0}+1 ))>0$ (if $\mu$ is supported on $(-\infty ,0)$ we can redefine $\mu \mapsto -\mu$). Since $|\!\cos(s)| \le 1$, for any domain $D\subset \mathbb{R}^{3}$ we have $\int_{\mathbb{R}^{3}}(\cos(\dots )-1)\le \int_{D}(\cos(\dots )-1).$ Hence, setting
 \[D= [m_{0},m_{0}+1)\times (\rho ,\infty )\times [0,2\pi ) ,\]
 in polar coordinates,  we have  
\begin{align}
\label{e:phibound}
  |\varphi (u,v)| 
\le &\exp\left(
\int_{m_{0}}^{m_{0}+1}\int_{\rho}^{\infty }\int_{0}^{2\pi }\left( \cos\left(
umh(r )+ \langle v,  mu_{\theta }h'(r) \rangle
\right)-1
\right) r drd\theta \mu (dm)
\right),
\end{align}
where $h(|x|) = g(x)$ (well-defined since $g$ is isotropic in all the cases we consider). For later use we note that $h$ and $h'$ are continuous and strictly decreasing on $(r_{0},\infty )$ for some $r_{0}>0$, and denote their inverse functions by $h^{\{-1\}},(h')^{\{-1\}}$. Recall the order-$0$ Bessel function, defined for $s \ge 0$ by
\[ J_{0}(s)=\frac{ 1}{2\pi }\int_{0}^{2\pi }\cos(s\cos(\theta ))d\theta=\frac{ 1}{2\pi }\int_{0}^{2\pi }\cos(s\sin(\theta ))d\theta . \]
Then, since $\int_{0}^{2\pi }\sin(s\cos(\theta ))d\theta =0$ for all $s\in \mathbb{R}$, the right-hand side of \eqref{e:phibound} equals
\begin{align*}
 &\exp \Big( \int_{m_{0}}^{m_{0}+1}
\int_{\rho}^{\infty } (  \cos(umh(r ))2\pi J_{0}(m|v||h'(r )|)-2\pi ) rdr\mu (dm) \Big)\\
& \qquad =\exp \bigg(
\underbrace{2\pi\int_{m_{0}}^{m_{0}+1}\int_{\rho}^{\infty } (\cos( umh(r ))-1)rdr \mu(dm)}_{I_{1}} \\
&\hspace{2.5cm} +\underbrace{2\pi  \int_{m_{0}}^{m_{0}+1}\int_{\rho}^{\infty }\cos{( umh(r ))}(J_{0}(m|v| | h'(r ) | )-1 )rdr\mu (dm)}_{I_{2}} \bigg).
\end{align*}
Let us bound $I_1$ and $I_2$ separately. In the rest of the proof, $\bar c,\bar c' > 0$ denote constants (depending only on $\alpha, \beta, m_{0}$) whose value might change from line to line. 

For $I_1$, recall that $\cos(s)-1\le -  s^{2}/3$ and $\cos(s)>1/2$ for $0<s \le s_{0}$, where $s_0$ is some positive constant. 
Define $u_{0} = s_0 / (m_0 + 1)>0$ and for $u\in \mathbb{R}$
\begin{align*}
\rho _{u} & =\inf\{\rho \ge r_{0} : |u|(m_{0}+1)h(r ) \le s_{0}  \text{ for all }  r \ge \rho  \} \\
&=h^{\{-1\}}(u_{0}/ | u | ).
\end{align*} 
We deduce that, for $\rho \ge \rho _{u}$,
\begin{align*}
I_{1}\le -\bar c |u|^{2}\int_{\rho }^{\infty }h(r)^{2}rdr.
\end{align*}
For $I_{2}$, we first recall that $J_{0}(t)-1 \le -c t^{2}$ for $0\le t\le t_{0}$, where $t_0$ is some positive constant, because $J_{0}$ is twice differentiable with $J_{0}(0)=1 ,J_{0}'(0)=0,J_{0}''(0)<0$.  Define $\rho '_{v},u_{0}'>0$ via
\[ \rho' _{v}=\inf\{\rho \ge r_{0} : |v|(m_{0}+1)h'(r) \le t_{0}  \text{ for all }  r\ge \rho \}=(h')^{\{-1\}}(u_{0}'/ | v | ). \] 
Then we have, for $\rho \ge \max\{\rho _{u},\rho' _{v}\}$,
\begin{align*}
 I_{2}
& \le -\bar c |v|^{2}\int_{\rho }^{\infty }h'(r)^{2}rdr.
\end{align*}

Recall that it is sufficient to prove that $|\varphi(u,v)|$ is integrable on $\mathbb{R}\times \mathbb{R}^{2}$. Let $u_{1}=1$ and $u_{1}'>0$ be such that $\rho _{u_{1}}=\rho '_{u_{1}'}>0$. We will actually prove that $a_{i}=\int_{D_{i}}|\varphi (u,v)|<\infty$ for $i=1,2$, where 
\begin{align*}
D_{1}=& \{u\in \mathbb{R},v\in \mathbb{R}^{2}:\;u_{1}'\le |v|,\;\rho _{u}\le \rho' _{v}\}\\
D_{2}=& \{u\in \mathbb{R},v\in \mathbb{R}^{2}:\;u_{1}\le |u|,\;\rho' _{v}\le \rho _{u}\}  
\end{align*}
and that $(D_{1}\cup D_{2})^{c}\subset D_{3}:=\{(u,v):  | u | \le u_{1},|v|\le u_{1}'\}$.
 For the latter point, let $(u,v)\notin D_{3}$. If $ | u | \le u_{1},$ then $|v|>u_{1}'$, and by strict monotonicity $$\rho _{u}\le \rho _{u_{1}}=\rho _{u_{1}'}'<\rho _{v}',$$  hence $(u,v)\in D_{2}$. If on the other hand $ | v | \le u_{1}',$ then similarly $ | u | >u_{1}$ and $\rho '_{v}<\rho _{u},$ hence $(u,v)\in D_{1}.$ 

It remains to show $a_{i}<\infty$ for $i=1,2$. Let us first treat the case $h(r)=(1+r)^{-\alpha }$, which yields
for $\rho \ge \max\{\rho _{u},\rho _{v}'\}$\begin{align*}
I_{1}\le -\bar cu^{2}\rho ^{2-2\alpha } \quad \text{and} \quad I_{2}\le -\bar c'|v|^{2}\rho ^{-2\alpha } .
\end{align*}

 For $|u|>u_{1} ,|v|>u_{1}'$, define 
\begin{align*}
\rho_{u,v}  =\begin{cases}
 \bar c'|u|^{1/\alpha }$ if $\rho _{u}>\rho _{v}'\\
 \bar c' | v | ^{\frac{1}{\alpha +1}}$ if $\rho _{v}'>\rho _{u}
\end{cases} 
\end{align*}
with $\bar c'$ chosen so that $\rho_{u,v} \ge \max\{\rho _{u},\rho _{v}'\}.$ We have
\[ |\varphi(u,v)|  \le  \exp(-\bar c u^{2}\rho_{u,v} ^{2-2\alpha } -\bar c' |v|^{2}\rho_{u,v} ^{-2\alpha }  ) .\]
Therefore 
\begin{align*}
a_{1} & \le \int_{D_{1}}\exp(-\bar c   |u|^{2 }\rho _{u,v}^{2-2\alpha }-\bar c  |v|^{2}\rho _{u,v}^{-2\alpha })\id_{\{\rho _{u,v}= \bar c' |v|^{\frac{1}{1+\alpha }}\}}dudv\\
& \le  \int_{ \{ | v | >u_{1}'\}}\exp(-\bar c  |u|^{2  }|v|^{\frac{2-2\alpha }{1+\alpha }}-\bar c  |v|^{2-2\frac{\alpha }{1+\alpha }}) dudv\\
a_{2} & \le \int_{D_{2}}\exp(-\bar c  |u|^{2 }\rho_{u,v} ^{2-2\alpha }-\bar c  |v|^{2}\rho_{u,v} ^{-2\alpha })\id_{\{\rho =\bar  c' |u|^{\frac{1}{\alpha }}\}}dudv\\
& \le \int_{ \{ | u | >u_{1}\}}\exp(-\bar c |u|^{2/\alpha } -\bar  c  |v|^{2}|u|^{-2})dudv.
\end{align*}
Since it is easy to check that these integrals are finite, the proof is complete in the case~\eqref{e:power}. The smoothed modification $h(r)=(1+r^{2})^{-\alpha /2}$ can be treated similarly as only the behaviours of $h$ and $h'$ at infinity matter.

In the case \eqref{e:exp},
define $\rho=\rho _{u,v} $  by
\[  \begin{cases}e ^{-\rho^{\beta }}= 
\bar c'/|u| $ \quad \quad \quad if $\rho _{u}>\rho _{v}\\
\rho ^{\beta -1}e^{-\rho  ^{\beta }}= \bar c'/|v| $  \quad otherwise$
\end{cases}\] 
with $\bar c'$ such that  $\rho>\bar c'\max\{\rho _{u},\rho _{v}\}.$ We have
\begin{align*}
I_{1} \le & -\bar cu^{2}\int_{\rho }^{\infty }re^{-2r^{ \beta }}dr \le -\bar c'u^{2}\rho ^{2-\beta }e^{-2\rho ^{\beta }}\\
I_{2} \le &-\bar c|v|^{2}\int_{\rho }^{\infty }r^{2\beta -1}e^{-2r^{\beta }}dr \le -\bar c'|v|^{2}\rho ^{\beta }e^{-2\rho ^{\beta }}
\end{align*}
as soon as  $|u|>u_{1}, |v|>u_{1}'$. 
On $D_{1}, e^{-\rho ^{\beta }} = \bar c'/|u|$, and so
\begin{align*}
a_{1}  & \le \int_{D_{1}}\exp(-\bar cu^{2}\rho ^{2-\beta }u^{-2}-\bar c|v|^{2}\rho ^\beta u^{-2})dudv\\
& \le \int_{ \{ | u | >u_{1}\} }\exp(-\bar c\rho ^{2-\beta }) \bigg(
\int_{0}^{\infty }\exp \bigg(
-\bar c \bigg(
\frac{|v|}{u\rho ^{-\beta /2}}
\bigg)^{2} \bigg)dv
\bigg) du \\
& \le \bar c\int_{ \{ | u | >u_{1}\} }\exp(-\bar c\rho ^{2-\beta })u \rho ^{-\beta /2}du.
\end{align*}
Since $2-\beta >\beta $ (recall we assume $\beta \in (0,1)$), 
\[ \exp(-\bar c\rho ^{2-\beta })u =\exp(-\bar c\rho ^{2-\beta }+\bar c'\rho ^{\beta })\le \bar  c'' \exp(-2\rho ^{\beta })=\bar c''u^{-2} , \]
and the term $\rho ^{-\beta /2}$ is logarithmic in $u$, hence $a_{1}<\infty$. 

On $D_{2},e^{-\rho ^{\beta }}=\bar c'\rho ^{1-\beta }/|v|$, and so
\begin{align*}
a_{2}  \le \int_{D_{2}}\exp(-\bar cu^{2}\rho ^{4-3\beta }|v|^{-2}-\bar c\rho ^{2-\beta })dudv \le \int_{\{ | v | >u_{1}'\} }\exp(-\bar c\rho ^{2-\beta })|v|\rho ^{-2+\frac{3\beta }{2}}dv.
  \end{align*}
As before, $\exp(-\bar c\rho^{2-\beta } )$ is dominated by $\exp(-\bar c'\rho^{\beta })$ for any $\bar c'>0$ and hence by any negative power of $|v|$, and the integral is finite.  Again the smoothed modification $h(r)=\exp(-(1+r^{2})^{\beta /2})$ can be treated similarly as only the behaviours of $h$ and $h'$ at infinity matter.
 \end{proof}

\bigskip

\bibliographystyle{halpha-abbrv}
\bibliography{sn}

\end{document}